\documentclass[11pt]{amsart}

\usepackage{amsmath}
\usepackage{amssymb}
\usepackage{graphicx}

\newtheorem{theorem}{Theorem}[section]
\newtheorem{corollary}[theorem]{Corollary}

\newtheorem{proposition}[theorem]{Proposition}
\theoremstyle{definition}

\numberwithin{equation}{section}
{\Large }

\newcommand{\R}{\ensuremath{\mathbb{R}}}

\newcommand{\N}{\ensuremath{\mathbb{N}}}

\newcommand{\cC}{\mathcal{C}}

\newcommand{\X}{\mathbb{X}}

\newcommand{\cD}{\mathcal{D}}
\newcommand{\cL}{\mathcal{L}}
\newcommand{\cM}{\mathcal{M}}
\newcommand{\bx}{\mathbf x}

\newcommand{\tp}{p{\rm -var}}
\newcommand{\tq}{q{\rm -var}}
\newcommand{\talpha}{\alpha{\rm -Hol}}
\newcommand{\tnu}{\nu{\rm -Hol}}

\newcommand{\cB}{\mathcal{B}}
\newcommand{\ty}{\bar y}

{\Large }

\newcommand{\ltn}{\ensuremath{\left| \! \left| \! \left|}}
\newcommand{\rtn}{\ensuremath{\right| \! \right| \! \right|}}

\title[Controlled differential equations as rough integrals]{Controlled differential equations \\as rough integrals}

\author[Luu Hoang Duc ]{Luu Hoang Duc }

\address[LUU HOANG DUC]{Max-Planck-Institute for Mathematics in the Sciences, Leipzig, Germany \& Institute of Mathematics, Viet Nam Academy of Science and Technology}
\email{{\tt duc.luu@mis.mpg.de, lhduc@math.ac.vn}}



\begin{document}

\begin{abstract}
We study controlled differential equations with unbounded drift terms, where the driving paths are $\nu$ - H\"older continuous for $\nu \in (\frac{1}{3},\frac{1}{2})$, so that the rough integral is interpreted in the Gubinelli sense \cite{gubinelli} for controlled rough paths. Similar to the rough differential equations in the sense of Lyons \cite{lyons98} or of Friz-Victoir \cite{friz}, we prove the existence and uniqueness theorem for the solution in the sense of Gubinelli, the continuity on the initial value, and the solution norm estimates. 
\end{abstract}

\maketitle


\section{Introduction}

This paper studies the following controlled differential equation
\begin{equation}\label{fSDE0}
dy_t =f(y_t)dt + g(y_t)d x_t,\quad \forall t\in [a,b],\ y_a \in \R^d,
\end{equation}
where the driving path $x$ belongs to a certain space of H\"older continuous functions. Such system is understood as a pathwise approach to solve a stochastic differential equation driven by a H\"older continuous stochastic process, for instance the path $x$ can be a realization of a fractional Brownian motion $B^H$  \cite{mandelbrot} with Hurst exponent $H \in (0,1)$. In this circumstance, equation \eqref{fSDE0} is often understood in the language of rough path theory, and its solutions are often defined in the sense of Lyons \cite{lyons98}, \cite{lyonsetal07}, or in the sense of Friz-Victoir \cite{friz}, \cite{BRSch17}. It is important to note that such definitions of rough differential equations do not specify what a rough integral is. 

On the other hand, in the simplest case of Young differential equations, the rough integrals are defined in the Young sense \cite{young}, and the existence and uniqueness theorem is well-known, see e.g. \cite{lyons94}, \cite{zahle}, \cite{nualart01}, \cite{congduchong17}. Recent results \cite{ducGANSch18}, \cite{duchong19} on the asymptotic stochastic stability and the existence of random attractors \cite{crauelkloeden} for the random dynamical systems \cite{arnold} generated by Young equations show a very effective method of applying the semigroup technique to estimate Young integrals and the discrete Gronwall lemma to derive the stability criterion \cite{duchong19}. To extend this method and stability results to rough differential equations, a first and necessary step is to define the rough integrals so that the rough system \eqref{fSDE0} can be understood in the integral form. Fortunately, this is feassible if one defines the rough integrals in the sense of Gubinelli \cite{gubinelli}, \cite{frizhairer} for controlled rough paths. An alternative approach is to define rough integral using fractional calculus, as studied for example in \cite{nualarthu}, \cite{GALSch16}, \cite{hessneamtu}.

Our aim in this paper is therefore to close the gap by proving similar results to \cite{BRSch17} for rough system with the unbounded drift term using rough integrals in the sense of Gubinelli. The main results of this paper are the existence and uniqueness theorem for the solution of \eqref{fSDE0}, the continuity of the solution on the initial condition, and the estimates of the solution norms.

To study the rough differential equation \eqref{fSDE0}, we impose the following assumptions.

(${\textbf H}_1$) $f: \R^d \to \R^d$ is globally Lipschitz continuous with the Lipschitz constant $C_f$;

(${\textbf H}_2$) $g$ either belongs to $C^3_b(\R^d,\cL(\R^m,\R^d))$ such that 
\begin{equation}\label{gcond}
\|g\|_\infty, C_g := \max\Big\{\|Dg\|_\infty, \|D_g^2\|_\infty, \|D^3_g\|_\infty \Big\} < \infty,
\end{equation}
or it has a linear form $g(y) = C y + g(0)$, where $C \in \cL(\R^d,\cL(\R^m,\R^d))$ such that $\|C\| \leq C_g$;

(${\textbf H}_3$) for a given $\nu \in (\frac{1}{3},\frac{1}{2})$, $x$ belongs to the space $\cC^{\nu}([a,b], \R^m)$ of all continuous paths which is of finite $\nu$ - H\"older norm on an interval $[a,b]$.

The paper is organized as follows. Section 2 is devoted to present a brief summary of rough path theory and the definition of rough integral for controlled rough paths in the sense of Gubinelli. Section 3 is devoted to the main results of this paper, namely we first give a simple proof in Theorem \ref{RDElinear} on the existence and uniqueness and norm estimates of the solution of \eqref{fSDE0} in the case $g$ of linear form. To prove the similar results for the nonlinear case, we prove Theorem \ref{prop2} on the differentiability of the solution of the system without the unbounded drift with respect to the initial value, and then apply the so-called {\it Doss-Sussmann technique} \cite{Sus78}, \cite{ImkSchm01} to transform equation \eqref{fSDE0} into an ordinary differential equation, thereby prove in Theorem \ref{RDEexist} the existence of solution of the original rough equation. The estimates of solution supremum and $p$ - variation norms are proved in Theorem \ref{RDEnorm}, which then helps to derive Theorem \ref{continuity} on the uniqueness and the continuity of the solution of \eqref{fSDE0} on the initial value.

\section{Rough integrals}
\subsection{Rough paths}
Let us present in this preparation section a short summary on rough path theory and rough integrals. Given any compact time interval $I = [\min I, \max I] \subset \R$, we write $|I| := \max{I} - \min{I}$ and $I^2 := I \times I$. For any finite dimensional vector space $W$, denote by $C(I,W)$ the space of all continuous paths $y: I \to W$ equipped with sup norm $\|\cdot\|_{\infty,I}$ given by $\|y\|_{\infty,I}=\sup_{t\in I} \|y_t\|$, where $\|\cdot\|$ is the norm in $W$. We write $y_{s,t}:= y_t-y_s$. For $p\geq 1$, denote by $C^{p{\rm-var}}(I,W)\subset C(I,W)$ the space of all continuous path $y:I \to W$ of finite $p$-variation 
\begin{eqnarray}
\ltn y\rtn_{p\text{-var},I} :=\left(\sup_{\Pi(I)}\sum_{i=1}^n \|y_{t_i,t_{i+1}}\|^p\right)^{1/p} < \infty,
\end{eqnarray}
where the supremum is taken over the whole class of finite partition of $I$. $C^{p{\rm-var}}(I,W)$ equipped with the $p-$var norm
\begin{eqnarray*}
	\|y\|_{p\text{-var},I}&:=& \|y_{\min{I}}\|+\ltn y\rtn_{p\rm{-var},I},
\end{eqnarray*}
is a nonseparable Banach space \cite[Theorem 5.25, p.\ 92]{friz}. Also for each $0<\alpha<1$, we denote by $C^{\alpha}(I,W)$ the space of H\"older continuous functions with exponent $\alpha$ on $I$ equipped with the norm
\begin{equation}\label{holnorm}
\|y\|_{\alpha,I}: = \|y_{\min{I}}\| + \ltn y\rtn_{\alpha,I},\quad \text{where} \quad \ltn y\rtn_{\alpha,I} :=\sup_{\substack{s,t\in I,\\ s<t}}\frac{\|y_{s,t}\|}{(t-s)^\alpha} < \infty.
\end{equation}
We now introduce the concept of rough paths. Following \cite{frizhairer}, a couple $\bx=(x,\X) \in \R^m \oplus (\R^m \otimes \R^m)$, where $x \in C^\alpha(I,\R^m)$ and 
\[
\X \in C^{2\alpha}(I^2,\R^m \otimes  \R^m):= \{\X \in C(I^2,\R^m \otimes  \R^m):  \sup_{\substack{s, t \in I,\\s<t}} \frac{\|\X_{s,t}\|}{|t-s|^{2\alpha}} < \infty \}, 
\]
is called a {\it rough path} if it satisfies Chen's relation
\begin{equation}\label{chen}
\X_{s,t} - \X_{s,u} - \X_{u,t} = x_{s,u} \otimes  x_{u,t},\qquad \forall \min{I} \leq s \leq u \leq t \leq \max{I}. 
\end{equation}
$\X$ is called a {\it L\'evy area} for $x$ and is viewed as {\it postulating} the value of the quantity $\int_s^t x_{s,r} \otimes  dx_r := \X_{s,t}$ where the right hand side is taken as a definition for the left hand side. Denote by $\cC^\alpha(I, \R^m \oplus (\R^m \otimes \R^m)) \subset C^\alpha(I,\R^m) \oplus C^{2\alpha}(I^2,\R^m \otimes \R^m)$ the set of all rough paths $\bx$ on $I$, then $\cC^\alpha(I)$ is a closed set (but not a linear space), equipped with the rough path semi-norm 
\begin{eqnarray}\label{translated}
\ltn \bx \rtn_{\alpha,I} &:=& \ltn x \rtn_{\alpha,I} + \ltn \X \rtn_{2\alpha,I^2}^{\frac{1}{2}},\quad \text{where}\qquad\\
 \ltn \X \rtn_{2\alpha,I^2}&:=& \sup_{s, t \in I;s<t} \frac{\|\X_{s,t}\|}{|t-s|^{2\alpha}} < \infty.  \notag
\end{eqnarray}
Throughout this paper, we will fix parameters $\frac{1}{3}< \alpha < \nu <\frac{1}{2}$ and  $p = \frac{1}{\alpha}$ so that $C^\alpha(I,W) \subset C^{\tp}(I,W)$. We also consider the $\tp$ semi-norm 
\begin{eqnarray}\label{pvarnorm}
\ltn \bx \rtn_{\tp,I} &:=& \Big(\ltn x \rtn^p_{\tp,I} + \ltn \X \rtn_{\tq,I^2}^q\Big)^{\frac{1}{p}}, \quad \text{for}\quad   q=\frac{p}{2} \quad \text{and} \\
\ltn \X \rtn_{\tq,I^2} &:=& \left(\sup_{\Pi(I)}\sum_{i=1}^n \|\X_{t_i,t_{i+1}}\|^q\right)^{1/q}, \notag
\end{eqnarray}
where the supremum is taken over the whole class of finite partition of $I$.  Sometimes, we write $\cC^\alpha(I)$ for abbreviation to neglect the value space for simplicity of presentation.

\subsection{Rough integrals}
We also introduce the construction of rough integrals. Following \cite{gubinelli}, a path $y \in C^\alpha(I,W)$ is called to be {\it controlled by} $x \in C^\alpha(I,\R^m)$ if there exists a tube $(y^\prime,R^y)$ with $y^\prime \in C^\alpha(I,\cL(\R^m,W)), R^y \in C^{2\alpha}(I^2,W)$ such that
\begin{equation}\label{controlRP}
y_{s,t} = y^\prime_s   x_{s,t} + R^y_{s,t},\qquad \forall \min{I}\leq s \leq t \leq \max{I}.
\end{equation}
$y^\prime$ is called Gubinelli derivative of $y$, which is uniquely defined as long as $x \in C^\alpha(I,\R^m)\setminus C^{2\alpha}(I,\R^m)$ (see \cite[Proposition 6.4]{frizhairer}). Denote by $\cD^{2\alpha}_x(I)$ the space of all the couple $(y,y^\prime)$ that is controlled by $x$, then $\cD^{2\alpha}_x(I)$ is a Banach space equipped with the norm
\begin{eqnarray*}
	\|(y,y^\prime)\|_{x,2\alpha,I} &:=& \|y_{\min{I}}\| + \|y^\prime_{\min{I}}\| + \ltn (y,y^\prime) \rtn_{x,2\alpha,I},\qquad \text{where} \\
	\ltn (y,y^\prime) \rtn_{x,2\alpha,I} &:=& \ltn y^\prime \rtn_{\alpha,I} +   \ltn R^y\rtn_{2\alpha,I^2},
\end{eqnarray*}
For a fixed rough path $\bx = (x,\X)$ and any controlled rough path $(y,y^\prime) \in \cD^{2\alpha}_x (I)$, it is proved in \cite{gubinelli} using the sewing lemma that the integral $\int_s^t y_u dx_u$ can be defined as  
\[
\int_s^t y_u dx_u := \lim \limits_{|\Pi| \to 0} \sum_{[u,v] \in \Pi} \Big( y_{u}  \otimes  x_{u,v} + y^\prime_u   \X_{u,v} \Big)
\]
where the limit is taken on all the finite partition $\Pi$ of $I$ with $|\Pi| := \displaystyle\max_{[u,v]\in \Pi} |v-u|$. Moreover, there exists a constant $C_\alpha = C_{\alpha,|I|} >1$, such that
\begin{eqnarray}\label{roughEst}
&& \Big\|\int_s^t y_u dx_u - y_s \otimes x_{s,t} - y^\prime_s \X_{s,t}\Big\|\notag\\
&&  \leq C_\alpha |t-s|^{3\alpha} \Big(\ltn x \rtn_{\alpha,[s,t]} \ltn R^y \rtn_{2\alpha,[s,t]^2} + \ltn y^\prime\rtn_{\alpha,[s,t]} \ltn \X \rtn_{2\alpha,[s,t]^2}\Big).
\end{eqnarray}
From now on, we sometimes simply write $\ltn x \rtn_{\alpha}$ or $\ltn \X \rtn_{2\alpha}$ without addressing the domain in $I$ or $I^2$. As proved in \cite{gubinelli}, the rough integral of controlled rough paths follows the rule of integration by parts.

In practice, we sometimes use the $p$-var norm
\begin{eqnarray*}
	\|(y,y^\prime)\|_{x,p,I} &:=& \|y_{\min{I}}\| + \|y^\prime_{\min{I}}\| + \ltn (y,y^\prime) \rtn_{x,p,I},\qquad \text{where} \\
	\ltn (y,y^\prime) \rtn_{x,p,I} &:=& \ltn y^\prime \rtn_{p{\rm -var},I} +   \ltn R^y\rtn_{\tq,I^2}.
\end{eqnarray*}
Thanks to the sewing lemma \cite{gubinelli}, we can use a similar version to \eqref{roughEst} under $p-$var norm as follows.
\begin{eqnarray}\label{roughpvar}
&& \Big\|\int_s^t y_u dx_u - y_s \otimes x_{s,t} -y^\prime_s \X_{s,t}\Big\| \notag\\
&& \leq C_p \Big(\ltn x \rtn_{\tp,[s,t]} \ltn R^y \rtn_{\tq,[s,t]^2} + \ltn y^\prime\rtn_{\tp,[s,t]} \ltn \X \rtn_{\tq,[s,t]^2}\Big),
\end{eqnarray}
with constant $C_p >1$ independent of $\bx$ and $y$. 

\subsection{Greedy sequence of stopping times}

Throughout this paper, we would need to use the concept of a greedy sequence of stopping times, as presented e.g. in \cite{cassetal}, \cite{congduchong17}, \cite{duchong19}. 
Given $\frac{1}{p}\in (\frac{1}{3},\nu)$, we construct for any fixed $\gamma \in (0,1)$ the sequence of greedy times $\{\tau_i(\gamma,I,p)\}_{i \in \N}$ w.r.t. the $p$ - var norm 
\begin{equation}\label{greedytime}
\tau_0 = \min{I},\quad \tau_{i+1}:= \inf\Big\{t>\tau_i:  \ltn \bx \rtn_{\tp, [\tau_i,t]} = \gamma \Big\}\wedge \max{I}.
\end{equation}
Denote by $N_{\gamma,I,p}(\bx):=\sup \{i \in \N: \tau_i \leq \max{I}\}$. It follows that
\begin{equation}\label{Nest}
N_{\gamma,I,p}(\bx) \leq 1 + \gamma^{-p} \ltn \bx \rtn^p_{\tp,I}.
\end{equation}
On the other hand, for $\alpha \in (\frac{1}{3},\nu)$, we also construct another greedy sequence of stopping times $\{\bar{\tau}_i(\gamma,I,\alpha)\}_{i \in \N}$ w.r.t. the $\alpha$ - H\"older norm
\begin{equation}\label{greedytimeHol}
\bar{\tau}_0 = \min{I},\quad \bar{\tau}_{i+1}:= \inf\Big\{t>\bar{\tau}_i:  (t-\bar{\tau}_i)^{1 - 2\alpha}+\ltn \bx \rtn_{\talpha, [\bar{\tau}_i,t]} = \gamma \Big\}\wedge \max{I},
\end{equation}
and assign $N_{\gamma,I,\alpha}(\bx):=\sup \{i \in \N: \bar{\tau}_i \leq \max{I}\}$. Then
\begin{equation}\label{Nholest}
N_{\gamma,I,\alpha}(\bx) \leq 1 + |I| \gamma^{-\frac{1}{\nu-\alpha}} \Big(1+\ltn \bx \rtn^{\frac{1}{\nu-\alpha}}_{\tnu,I}\Big).
\end{equation}

\section{Existence-uniqueness theorem and solution norm estimates}

In this section, we would like to prove the existence and uniqueness theorem and estimate solution norms for rough differential equation \eqref{fSDE0}, where the rough integral is understood in the sense of Gubinelli \cite{gubinelli} for controlled rough paths. In case the diffusion coefficient $g$ is linear, the proof is straight forward as presented below, which applies Schauder-Tichonorff theorem.

\begin{theorem}\label{RDElinear}
	There exists a unique solution of the rough differential equation 
	\begin{equation}\label{Existlinear1}
	dy_t =  f(y_t)dt + \Big(C y_t + g(0)\Big) d x_t,\quad \forall t\in [a,b],\ y_a \in \R^d. 
	\end{equation}
	Moreover, the supremum and $p-$variation norms of the solution are estimated as follows
	\begin{eqnarray}\label{estxlin}
	\|y\|_{\infty,[a,b]} &\leq&  \Big[\|y_a\| + M_0 N_{[a,b]}(\bx)\Big] e^{4C_f (b-a) + L N_{[a,b]}(\bx)}, \notag \\
	\ltn y,R^y \rtn_{p{\rm -var},[a,b]} &\leq& \Big[\|y_a\| + M_0  N_{[a,b]}(\bx)\Big] \times \\
	&& \times e^{4C_f(b-a) + L N_{[a,b]}(\bx)}N^{\frac{p-1}{p}}_{[a,b]}(\bx) -\|y_a\|,\notag
	\end{eqnarray}
	where $\ltn y,R^y \rtn_{\tp,[s,t]} :=\ltn y \rtn_{\tp,[s,t]} + \ltn R^y \rtn_{\tq,[s,t]^2}$, $M_0=(1+\frac{3}{2C_p}) \frac{\|g(0)\|}{\|C\|}+\frac{\|f(0)\|}{C_f}$ and $L = \log(1+ \frac{3}{2C_p})$. 
\end{theorem}
\begin{proof}
	{\bf Step 1}. Rewrite equation \eqref{RDElinear} in the integral form
	\begin{equation}\label{linRDE3}
	y_t = G(y,y^\prime)_t = y_a + \int_a^t f(y_u) du + \int_a^t \Big(C y_u +g(0) \Big)dx_u,\qquad t\in [a,T],
	\end{equation}
	Denote by $\cD^{2\alpha}_x\Big(y_a,C y_a+g(0)\Big)$  the set of paths $(y,y^\prime)$ controlled by $x$ in $[a,T]$ with fixed initial conditions $y_a$ and $y^\prime_a = Cy_a + g(0)$. Consider the mapping defined by
	\begin{eqnarray*}
	\cM: \cD^{2\alpha}_x\Big(y_a,C y_a+g(0)\Big) &\to& \cD^{2\alpha}_x\Big(y_a,Cy_a+g(0)\Big),\\
	\cM (y,y^\prime)_t &:=& \Big(G(y,y^\prime)_t, Cy_t+g(0)\Big). 
	\end{eqnarray*}
	Similar to \cite{gubinelli} we are going to estimate $\ltn \cM (y,y^\prime)\rtn_{x,2\alpha} = \ltn C y +g(0)\rtn_\alpha + \ltn R^{G(y,y^\prime)} \rtn_{2\alpha}$ using $\ltn(y,y^\prime)\rtn_{x,2\alpha} = \ltn y^\prime \rtn_\alpha + \ltn R^y \rtn_{2\alpha}$. Observe that
	\begin{eqnarray*}
		\ltn C y +g(0)\rtn_\alpha &\leq& \|C\| \ltn y \rtn_\alpha \leq  \|C\| \Big(\|y^\prime\|_{\infty} \ltn x\rtn_\alpha + (T-a)^\alpha \ltn R^y \rtn_{2\alpha}\Big)\\
		&\leq& \|C\| \ltn x \rtn_\alpha \|y^\prime_a\| + \|C\|(T-a)^\alpha \ltn x \rtn_\alpha \ltn y^\prime \rtn_\alpha \\
		&& + \|C\| (T-a)^\alpha \ltn R^y \rtn_{2\alpha},
	\end{eqnarray*}
	and
	\begin{eqnarray*}
		\|R^{G(y,y^\prime)}_{s,t}\| &\leq&  \Big\|\int_s^t f(y_u) du\Big\| + \Big\|\int_s^t C y_u dx_u - Cy_s \otimes x_{s,t}\Big\| \\
		&\leq&  C_f|t-s| \|y\|_{\infty} + \|f(0)\|(t-s)+ \|C\| \|y^\prime\|_{\infty} |\X_{s,t}| \\
		&&+ C_\alpha |t-s|^{3\alpha} \Big[\ltn x \rtn_{\alpha} \|C\| \ltn R^y \rtn_{2\alpha} + \|C\| \ltn y^\prime \rtn_{\alpha} \ltn \X \rtn_{2\alpha} \Big],
	\end{eqnarray*}
	where we can choose $T-a <1$ so that $C_\alpha$ can be bounded from above by $C_{\alpha,1}$. Since 
	\[
	\|y\|_{\infty,[s,t]} \leq \|y_a\| + \|y^\prime_a\| (T-a)^{\alpha} \ltn x \rtn_{\alpha} + (T-a)^{2 \alpha} \ltn R^y \rtn_{2 \alpha},
	\]
	it follows that
	\begin{eqnarray*}
		&&\ltn R^{G(y,y^\prime)} \rtn_{2\alpha} \\
		&\leq& (T-a)^{1-2\alpha}( C_f \|y_a\|+\|f(0)\|) + (T-a)^{1-\alpha} C_f \ltn x \rtn_\alpha \|y_a^\prime\| \\
		&&+ C_f(T-a)\ltn R^y \rtn_{2\alpha} + \|C\| \ltn \X \rtn_{2\alpha} (|y_a^\prime|+ (T-a)^\alpha \ltn y^\prime \rtn_\alpha)\\
		&&+ C_\alpha \|C\| (T-a)^\alpha \Big[\ltn x \rtn_{\alpha} \ltn R^y \rtn_{2\alpha} + \ltn y^\prime \rtn_{\alpha} \ltn \X \rtn_{2\alpha} \Big] 
	\end{eqnarray*}
	Hence, we can estimate $\ltn \cM (y,y^\prime)\rtn_{x,2\alpha}$ as follows
	\allowdisplaybreaks
	\begin{eqnarray*}
		&&\ltn \cM(y,y^\prime) \rtn_{x, 2\alpha} \\
		&\leq& \|C\|\Big[(\|y^\prime_a\| + (T-a)^\alpha \ltn y^\prime \rtn_\alpha)\ltn x \rtn_\alpha + (T-a)^\alpha \ltn R^y \rtn_{2\alpha} \Big] + \ltn R^{G(y,y^\prime)} \rtn_{2\alpha}\\
		&\leq& (T-a)^{1-2\alpha}(C_f \|y_a\|  +\|f(0)\|)\\
		&&+ \Big[\Big(\|C\| + (T-a)^{1-\alpha} C_f \Big) \ltn x \rtn_{\alpha}+ \|C\| \ltn \X \rtn_{2\alpha} \Big] \|C y_a+g(0)\| \\
		&&+ \Big[(T-a)^\alpha \|C\| \ltn x \rtn_\alpha + (T-a)^\alpha \|C\| (1+ C_\alpha) \ltn \X \rtn_{2\alpha} \Big] \ltn y^\prime \rtn_\alpha \\
		&&+ \Big[\|C\| (T-a)^\alpha + (T-a)C_f + C_\alpha \|C\| (T-a)^\alpha \ltn x \rtn_\alpha \Big] \ltn R^y \rtn_{2 \alpha}\\
		&\leq& \Big(C_f + \|C\| + \|C\| C_\alpha\Big) (1 + \|C\|) \frac{\mu}{2M} \Big(\|y_a\| + \frac{\|g(0)\|}{\|C\|}+\frac{\|f(0)\|}{C_f}\Big) \\
		&& + \Big[C_f + \|C\| + \|C\| C_\alpha \Big] \frac{\mu}{2M} \Big(\ltn y^\prime \rtn_\alpha + \ltn R^y \rtn_{2\alpha}\Big)\\
		&\leq& \mu \Big(\|y_a\|+\frac{\|g(0)\|}{\|C\|}+\frac{\|f(0)\|}{C_f}+\ltn (y,y^\prime)\rtn_{x,2\alpha}\Big)
	\end{eqnarray*}
	where we choose for a fixed number $\mu \in (0,1)$ with
	\[
	M := \max\Big\{ \Big[C_f + \|C\|(1 +C_\alpha)\Big](1+\|C\|), \frac{1}{2}\Big\}
	\]
	and $T = T(a)$ satisfying
	\begin{equation*}
	(T-a)^{1-2\alpha} + \ltn x \rtn_{\alpha, [a,T]} + \ltn \X \rtn_{2\alpha, [a,T]^2}^{\frac{1}{2}} = \frac{\mu}{2M}<1. 
	\end{equation*}
	As a result, if we restrict to the convex compact set
	\[
	\cB := \Big\{ (y,y^\prime) \in \cD^{2\alpha}_x(y_a,C y_a), \ltn (y,y^\prime)\rtn_{x,2\alpha} \leq \frac{\mu}{1- \mu} \Big(\|y_a\| + \frac{\|g(0)\|}{\|C\|}+\frac{\|f(0)\|}{C_f}\Big) \Big\}
	\]
	then
	\begin{eqnarray*}
		\ltn \cM(y,y^\prime) \rtn_{x, 2\alpha} &\leq &\mu \Big(\|y_a\|+\frac{\|g(0)\|}{\|C\|}+\frac{\|f(0)\|}{C_f}\Big) + \mu \|(y,y^\prime)\|_{x,2\alpha} \\
		&\leq&\Big(\frac{\mu^2}{1- \mu}+\mu\Big)\Big(\|y_a\|+\frac{\|g(0)\|}{\|C\|}+\frac{\|f(0)\|}{C_f}\Big) \\
		&\leq& \frac{\mu}{1-\mu}\Big(\|y_a\|+\frac{\|g(0)\|}{\|C\|}+\frac{\|f(0)\|}{C_f}\Big),
	\end{eqnarray*}
	which proves that $\cM: \cB \to \cB$. By Schauder-Tichonorff theorem, there exists a  fixed point of $\cM$ which is a solution of equation \eqref{Existlinear1} on the interval $[a,T]$. 
	
	Next, for any two solutions $(y,y^\prime), (\bar{y},\bar{y}^\prime)$ of the same initial conditions $(y_a, Cy_a)$, by similar computations, we obtain
	\begin{eqnarray*}
		\ltn (y,y^\prime) - (\bar{y},\bar{y}^\prime)\rtn_{x,2\alpha}
		&\leq& \mu \Big(\|y_a-\bar{y}_a\| + \ltn (y,y^\prime) - (\bar{y},\bar{y}^\prime) \rtn_{x,2\alpha}\Big)\\ &\leq&  \mu \ltn (y,y^\prime)-(\bar{y},\bar{y}^\prime)\rtn_{x,2\alpha}, 
	\end{eqnarray*}
	which, together with $\mu <1$, proves the uniqueness of solution of \eqref{Existlinear1} on $[a,T]$. By constructing the greedy sequence of stopping times $\{\bar{\tau}_i(\frac{\mu}{2M},I,\alpha)\}_{i \in \N}$ as in \eqref{greedytimeHol}, we can extend and prove the existence of the unique solution on the interval $[a,b]$. 
	
	{\bf Step 2}. To estimate the norms, note that
	\begin{eqnarray*}
		y^\prime_s = Cy_s + g(0),\quad [Cy + g(0)]_s^\prime = C^2 y_s + C g(0),\quad R_{s,t}^{Cy+g(0)} = C R_{s,t}^y. 
	\end{eqnarray*}
	The estimate $\|y_{s,t}\|$ is direct using \eqref{roughpvar}, namely
	\begin{eqnarray}\label{linest1}
	\|y_{s,t}\| &\leq& \int_s^t (C_f \|y_u\| + \|f(0)\|) du \notag\\
	&& + (\|C\| \|y_s\| + \|g(0)\|) (\|x_{s,t}\|+ \|C\|\|\X_{s,t}\|) \notag\\
	&& + C_p \Big(\ltn x\rtn_{\tp,[s,t]} \ltn R^{Cy+g(0)}\rtn_{\tq,[s,t]^2} \notag\\
	&& + \ltn \X \rtn_{\tq,[s,t]^2} \ltn [Cy + g(0)]^\prime \rtn_{\tp,[s,t]} \Big) \notag\\
	&\leq& \int_s^t (C_f \|y\|_{\tp,[s,u]} + \|f(0)\|) du \notag\\
	&& + \Big(\|y_s\| + \frac{\|g(0)\|}{\|C\|}\Big) (\|C\|\|x_{s,t}\|+\|C\|^2 \|\X_{s,t}\|) \notag\\
	&& + C_p \Big(\ltn x\rtn_{\tp,[s,t]} \|C\| \ltn R^{y}\rtn_{\tq,[s,t]^2} \notag\\
	&&+ \ltn \X \rtn_{\tq,[s,t]^2} \|C\|^2 \ltn y^\prime \rtn_{\tp,[s,t]} \Big) \notag\\
	&\leq& \int_s^t (C_f \|y\|_{\tp,[s,u]} + \|f(0)\|) du \notag\\
	&&+ 2\Big(\|y_s\| + \frac{\|g(0)\|}{\|C\|}\Big) (\|C\| \|x_{s,t}\| \vee \|C\|^2\|\X_{s,t}\|) \notag\\
	&& + C_p \Big\{\|C\|^2\ltn \X \rtn_{\tq,[s,t]^2} \vee \|C\| \ltn x \rtn_{\tp,[s,t]} \Big\} \ltn y,R^y \rtn_{\tp,[s,t]}.
	\end{eqnarray}
	The estimate for $\|R^y_{s,t}\|$ is already included in \eqref{linest1}. It follows that
	\begin{eqnarray*}
	&&	\ltn y,R^y \rtn_{\tp,[s,t]} \\
		&\leq& 2\int_s^t (C_f \ltn y\rtn_{\tp,[s,u]} +C_f\|y_s\|+ \|f(0)\|)du \\
		&& + 3\Big(\|C\| \ltn x \rtn_{\tp,[s,t]}  \vee \|C\|^2 \ltn \X \rtn_{\tq,[s,t]^2} \Big) \Big(\|y_s\| +\frac{\|g(0)\|}{\|C\|} \Big) \\
		&&+2C_p \Big\{\|C\|^2\ltn \X \rtn_{\tq,[s,t]^2} \vee \|C\| \ltn x \rtn_{\tp,[s,t]} \Big\} \ltn y,R^y \rtn_{\tp,[s,t]}.
	\end{eqnarray*}
	As a result
	\begin{eqnarray*}
	\ltn y,R^y \rtn_{p{\rm -var},[s,t]} &\leq& \int_s^t 4 C_f \ltn y,R^y \rtn_{p{\rm -var},[s,u]} du \\
	&&+ 4(\|f(0)\|+ C_f \|y_s\|)(t-s) +\frac{3}{2C_p} \Big(\|y_s\| + \frac{\|g(0)\|}{\|C\|}\Big)
	\end{eqnarray*}
	whenever $2C_p \|C\| \ltn \bx \rtn_{\tp,[s,t]} \leq \frac{1}{2}$. Applying the continuous Gronwall lemma, we obtain
	\allowdisplaybreaks
	\begin{eqnarray}\label{yqvarest}
	\ltn y,R^y \rtn_{p{\rm -var},[s,t]} &\leq&  4(\|f(0)\|+ C_f \|y_s\|)(t-s) + \frac{3}{2C_p} \frac{\|g(0)\|}{\|C\|}+ \frac{3}{2C_p} \|y_s\| \notag\\
	&&+ \int_s^t 4C_f e^{4C_f (t-u)} \Big[4\Big(\|f(0)\|+ C_f \|y_s\|\Big)(u-s) \notag\\
	&&+ \frac{3}{2C_p} \frac{\|g(0)\|}{\|C\|} 
	+ \frac{3}{2C_p} \|y_s\|\Big] du \notag\\
	&\leq& \Big(M_0 + e^L\|y_s\| \Big)e^{4C_f(t-s)} - \|y_s\|
	\end{eqnarray}
	whenever $4C_p \|C\| \ltn \bx \rtn_{\tp,[s,t]} \leq 1$. By constructing the greedy sequence of stopping times $\{\tau_i(\frac{1}{4C_p\|C\|},[a,b],p) \}$ and use similar estimates to the proof of \cite[Theorem 2.4]{duchong19}, we obtain \eqref{estxlin}.
	
\end{proof}

\begin{corollary}
	The solution of \eqref{Existlinear1} is uniformly continuous in $y_a$, i.e. for any two solutions $y_\cdot(\bx,y_a)$ and $\ty_\cdot(\bx,\ty_a)$, the following estimates hold 
	\begin{eqnarray}\label{estxlin2}
	\|\ty-y\|_{\infty,[a,b]} &\leq&  \|\ty_a-y_a\| e^{4C_f (b-a) + L N_{[a,b]}(\bx)}, \notag \\
	\ltn \ty-y,R^{\ty-y} \rtn_{p{\rm -var},[a,b]} &\leq& \|\ty_a-y_a\| e^{4C_f(b-a) + L N_{[a,b]}(\bx)}N^{\frac{p-1}{p}}_{[a,b]}(\bx) \\
	&& -\|\ty_a-y_a\|.\notag
	\end{eqnarray}
\end{corollary}
\begin{proof}
	The proof follows similar arguments and will be omitted. 	
\end{proof}

Unlike the linear case, the proof for the nonlinear case $g \in C^3_b$ is not simple by just following the scheme in the linear case. This is because the estimate of the rough integral $\int_s^t g(y_u)dx_u$ depends on the H\"older seminorm $\ltn [g(y)]^\prime\rtn_{\alpha}$, which would be estimated by a non-linear term $\ltn y \rtn_{\alpha} \|y^\prime\|_\infty$. As a result, if one would attempt to apply the Schauder-Tichonorff theorem directly, then the possible greedy sequence of stopping times constructed in Theorem \ref{RDElinear} would depend on the initial value $\|y_a\|$ and it might not be easy to extend the solution to any interval. To overcome this difficulty, we propose another scheme, by first proving the differentiability w.r.t. the initial condition of the solution of the rough differential equation
\begin{equation}\label{RDEpart}
dy_t = g(y_t) dx_t, \quad \forall t \in [a,b], y_a \in \R^d,
\end{equation}
and then applying the Doss-Sussmann technique \cite{Sus78}, \cite{ImkSchm01} to transform the original system to an ordinary differential equation. Note that the existence, uniqueness and continuity of the solution of \eqref{RDEpart} is already provided in \cite{gubinelli}, but the differentiability of the solution $y_t(\bx,y_a)$ w.r.t. $y_a$ is somehow missing due to the technical complexity. We will provide a detailed proof for this assertment. First, we need the following estimates.

\begin{proposition}\label{prop1}
	The solution $y_t(\bx,y_a)$ of \eqref{RDEpart} is uniformly continuous w.r.t. $y_a$, i.e. for any two solutions $y_\cdot(\bx,y_a)$ and $\ty_\cdot(\bx,\ty_a)$ the following estimates hold
	\begin{eqnarray}\label{ydiff}
	\|\ty -y \|_{\infty,[a,b]} &\leq& \|\ty_a -y_a\|e^{(\log 2) \bar{N}_{[a,b]}(\bx)},  \notag\\
	\ltn \ty-y,R^{\ty-y} \rtn_{\tp,[a,b]} &\leq& \|\ty_a -y_a\| \bar{N}^{\frac{p-1}{p}}_{[a,b]}(\bx)e^{(\log 2) \bar{N}_{[a,b]}(\bx)} \\
	&&- \|\ty_a -y_a\|,\notag
	\end{eqnarray}
	where $\bar{N}_{[a,b]}(\bx)$ is the maximal index of the maximal stopping time in the greedy sequence 
	\begin{equation}\label{greedybarN}
	\tau_0 = a, \tau_{k+1} := \inf \Big\{t > \tau_k: 8C_pC_g C(\bx,[a,b])\ltn \bx \rtn_{\tp,[\tau_k,t]} =1 \Big\} \wedge b,
	\end{equation}
	that lies in the interval $[a,b]$, with 
	\[
	C(\bx,[a,b]):= \Big\{1 + \frac{1}{C_p} 2^{\frac{2p-1}{p}} \Big(1+ [8C_pC_g]^{2p-1} \ltn \bx \rtn_{\tp,[a,b]}^{2p-1}\Big)\Big\}.
	\]	
\end{proposition}
\begin{proof}
	The proof is divided into several steps.\\
	
	{\bf Step 1:} To estimate the solution norms of \eqref{RDEpart}, observe that due to \eqref{roughEst}
	\begin{eqnarray*}
	&& \Big\|\int_s^t g(y_u) dx_u - g(y_s) \otimes x_{s,t} - [g(y)]^\prime_s  \X_{s,t} \Big\| \\
	&\leq& C_\alpha (t-s)^{3\alpha}\Big[\ltn x \rtn_{\alpha,[s,t]} \ltn R^{g(y)} \rtn_{2\alpha,[s,t]^2} + \ltn \X \rtn_{2\alpha,[s,t]^2} \ltn [g(y)]^\prime \rtn_{\alpha,[s,t]}\Big].  
	\end{eqnarray*}
	It follows that $y$ is controlled by $x$ with $y^\prime = g(y)$. Since
	\begin{eqnarray*}
		g(y_t) - g(y_s)
		&=& \int_0^1 Dg(y_s + \eta y_{s,t}) y_{s,t}d\eta \\
		&=& D_g(y_s) y^\prime_s \otimes  x_{s,t} + \int_0^1 Dg(y_s + \eta y_{s,t}) R^y_{s,t}d\eta \\
		&& + \int_0^1 [Dg(y_s + \eta y_{s,t}) - Dg(y_s)] y^\prime_{s,t} \otimes x_{s,t}d\eta,
	\end{eqnarray*}
	it easy to show that $[g(y)]^\prime_s = Dg(y_s)g(y_s)$, where we use \eqref{gcond} to estimate
	\begin{eqnarray*}
		\|R^{g(y)}_{s,t}\| &\leq& \int_0^1 \|Dg(y_s + \eta y_{s,t})\| \|R^y_{s,t}\|d\eta \\
		&& + \int_0^1 \|Dg(y_s + \eta y_{s,t}) - Dg(y_s)\| \|g(y_s)\| \|x_{s,t}\|d\eta \\
		&\leq& C_g \|R^y_{s,t}\| + \frac{1}{2} C_g^2 \|y_{s,t}\| \|x_{s,t}\|.
	\end{eqnarray*}
	This together with H\"older inequality yields
	\allowdisplaybreaks
	\begin{eqnarray}\label{Exist2}
	\ltn [g(y)]^\prime \rtn_{\tp,[s,t]} &\leq& 2 C_g^2 \ltn y \rtn_{\tp,[s,t]}, \quad \|[g(y)]^\prime\|_{\infty,[s,t]} \leq C_g^2, \notag\\
	\ltn R^{g(y)} \rtn_{\tq,[s,t]^2} &\leq& C_g \ltn R^y \rtn_{\tq,[s,t]^2} \\
	&& + \frac{1}{2}C_g^2 \ltn x \rtn_{\tp,[s,t]} \ltn y \rtn_{\tp,[s,t]}.\notag
	\end{eqnarray}
	As a result, by introducing $\ltn y,R^y \rtn_{\tp,[s,t]} :=\ltn y \rtn_{\tp,[s,t]} + \ltn R^y \rtn_{\tq,[s,t]^2}$, we obtain
	\begin{eqnarray*}
		&&\|y_{s,t}\| \\
		&\leq& \Big\| \int_s^t g(y_u) dx_u\Big\| \\
		&\leq& \|g(y_s)\| \|x_{s,t}\| + \|Dg(y_s)g(y_s)\| \|\X_{s,t}\| \\
		&&+ C_p \Big[ \ltn x \rtn_{\tp,[s,t]} \ltn R^{g(y)}\rtn_{\tq,[s,t]^2} + \ltn \X \rtn_{\tq,[s,t]^2} \ltn [g(y)]^\prime \rtn_{\tp,[s,t]}\Big]  \\
		&\leq& C_g \ltn x \rtn_{\tp,[s,t]} + C_g^2 \ltn \X \rtn_{\tq,[s,t]^2} + C_p \Big\{C_g \ltn x \rtn_{\tp,[s,t]} \ltn R^y \rtn_{\tq,[s,t]^2} \\
		&&+ \Big[\frac{1}{2} \ltn x\rtn_{\tp,[s,t]}^2 + 2 \ltn \X \rtn_{\tq,[s,t]^2}\Big] C_g^2 \ltn y \rtn_{\tp,[s,t]} \Big \} \\
		&\leq& C_g \ltn \bx \rtn_{\tp,[s,t]} + C_g^2 \ltn \bx \rtn_{\tp,[s,t]}^2 \\
		&&+ 2C_p \Big\{C_g \ltn \bx \rtn_{\tp,[s,t]} \vee  C_g^2 \ltn \bx \rtn_{\tp,[s,t]}^2  \Big\} \ltn y, R^y\rtn_{\tp,[s,t]},\\
		&\leq& 2\Big\{C_g \ltn \bx \rtn_{\tp,[s,t]} \vee  C_g^2 \ltn \bx \rtn_{\tp,[s,t]}^2  \Big\} \Big(1 + C_p\ltn y, R^y\rtn_{\tp,[s,t]}\Big),
	\end{eqnarray*}
	which derives
	\begin{eqnarray}\label{ypvar} 
	\ltn y \rtn_{\tp,[s,t]} &\leq& 2\Big\{C_g \ltn \bx \rtn_{\tp,[s,t]} \vee  C_g^2 \ltn \bx \rtn_{\tp,[s,t]}^2  \Big\} \times \notag\\
	&& \times \Big(1 + C_p\ltn y, R^y\rtn_{\tp,[s,t]}\Big).
	\end{eqnarray}
	The same estimate for $R^y$ is actually included in the above estimate, hence
	\begin{eqnarray}\label{Rqvar}
	\ltn R^y \rtn_{\tq,[s,t]^2} &\leq& 2\Big\{C_g \ltn \bx \rtn_{\tp,[s,t]} \vee  C_g^2 \ltn \bx \rtn_{\tp,[s,t]}^2  \Big\} \times\notag\\
	&& \times \Big(1 + C_p\ltn y, R^y\rtn_{\tp,[s,t]}\Big).
	\end{eqnarray}
	Combining \eqref{ypvar} and \eqref{Rqvar} gives
	\begin{eqnarray}\label{yRpvar}
	\ltn y,R^y \rtn_{\tp,[s,t]} &\leq& 4\Big\{C_g \ltn \bx \rtn_{\tp,[s,t]} \vee  C_g^2 \ltn \bx \rtn_{\tp,[s,t]}^2  \Big\} \times \notag\\
	&& \times \Big(1 + C_p\ltn y, R^y\rtn_{\tp,[s,t]}\Big).
	\end{eqnarray}
	It implies from \eqref{yRpvar} that 
	\begin{eqnarray*}
	&& \ltn y,R^y \rtn_{\tp,[s,t]} \leq \frac{1}{C_p}\\
	&\text{whenever}& \Big\{C_g \ltn \bx \rtn_{\tp,[s,t]} \vee  C_g^2 \ltn \bx \rtn_{\tp,[s,t]}^2  \Big\} \leq \frac{1}{8C_p} <1, 
	\end{eqnarray*}
	which yields
	\[
	\ltn y,R^y \rtn_{\tp,[s,t]} \leq \frac{1}{C_p} \quad \text{whenever} \quad \ltn \bx \rtn_{\tp,[s,t]} \leq \frac{1}{8C_pC_g}.
	\]
	Using similar arguments to \cite[Theorem 2.4]{duchong19}, by constructing a greedy sequence of stopping times $\{\tau_i(\frac{1}{8C_pC_g},I,p)\}_{i \in \N}$ as in \eqref{greedytime}, we conclude that
	\begin{eqnarray}
	\|y\|_{\infty,[a,b]} &\leq& \|y_a\| + \frac{1}{C_p}N_{[a,b]}(\bx) \notag\\
	&\leq& \|y_a\| + \frac{1}{C_p} \Big(1+[8C_pC_g]^p \ltn \bx \rtn_{\tp,[a,b]}^p\Big), \label{ysupest} \\
	\ltn y,R^y \rtn_{p{\rm - var},[a,b]} &\leq& N^{\frac{p-1}{p}}_{[a,b]}(\bx)\sum_{k = 0}^{N_{[a,b]}(\bx)-1} \ltn y,R^y \rtn_{q{\rm -var},[\tau_k,\tau_{k+1}]}\notag\\
	&\leq& \frac{1}{C_p}N^{\frac{2p-1}{p}}_{[a,b]}(\bx), \notag\\
	&\leq& \frac{1}{C_p} 2^{\frac{p-1}{p}} \Big(1+ [8C_pC_g]^{2p-1} \ltn \bx \rtn_{\tp,[a,b]}^{2p-1}\Big),
	\label{yRpvarest}
	\end{eqnarray}
	where the last inequality applies \eqref{Nest} and the H\"older inequality.
	
	{\bf Step 2:} Next, for any two solutions $y_t(\bx,y_a)$ and $\ty_t(\bx,\ty_a)$ within the bounded range $\frac{1}{C_p}N^{\frac{2p-1}{p}}_{[a,b]}(\bx)$, consider their difference $z_t = \ty_t -y_t$, which satisfies the integral rough equation
	\[
	z_t = z_a +\int_a^t [g(\ty_s) - g(y_s)]dx_s.
	\]
	As a result, $y^\prime_s = g(y_s), \ty_s^\prime = g(\ty_s)$ and
	\allowdisplaybreaks
	\begin{eqnarray*}
		&&g(\ty_t) - g(y_t) - g(\ty_s) + g(y_s) \notag\\
		&=& \int_0^1 \Big[ Dg(\ty_s + \eta \ty_{s,t}) \ty_{s,t} - Dg(y_s + \eta y_{s,t})y_{s,t}\Big]d\eta \notag\\
		&=&\Big[ Dg(\ty_s)g(\ty_s) - Dg(y_s)g(y_s) \Big] \otimes x_{s,t}\\
		&&+ \int_0^1 \Big[Dg(\ty_s + \eta \ty_{s,t})R^{\ty}_{s,t} - Dg(y_s + \eta y_{s,t})R^y_{s,t}\Big] d\eta \notag\\
		&&+ \int_0^1 \Big\{\Big[Dg(\ty_s + \eta \ty_{s,t}) - Dg(\ty_s)\Big]g(\ty_s)\\
		&& - \Big[Dg(y_s + \eta y_{s,t}) - Dg(y_s)\Big] g(y_s)\Big\} \otimes x_{s,t} d\eta \notag\\
		&=&\Big[ Dg(\ty_s)g(\ty_s) - Dg(y_s)g(y_s) \Big] \otimes x_{s,t}\\
		&& + \int_0^1 \Big\{Dg(\ty_s + \eta \ty_{s,t})\Big(R^{\ty}_{s,t}-R^y_{s,t}\Big) \\
		&&+ \Big[Dg(\ty_s + \eta \ty_{s,t})- Dg(y_s + \eta y_{s,t})\Big] R^y_{s,t}\Big\}d\eta \notag\\
		&&+ \int_0^1 \Big[Dg(\ty_s + \eta \ty_{s,t}) - Dg(\ty_s)\Big] \Big[g(\ty_s)-g(y_s)\Big] \otimes x_{s,t}d\eta \notag\\
		&&+ \int_0^1 \Big[Dg(\ty_s + \eta \ty_{s,t}) - Dg(\ty_s) -Dg(y_s + \eta y_{s,t}) + Dg(y_s)\Big] g(y_s) \otimes x_{s,t} d\eta \notag \\
		&=&\Big[ Dg(\ty_s)g(\ty_s) - Dg(y_s)g(y_s) \Big] \otimes x_{s,t}  \notag\\
		&& + \int_0^1 \Big\{Dg(\ty_s + \eta \ty_{s,t}) R^{\ty-y}_{s,t} + \Big[Dg(\ty_s + \eta \ty_{s,t})- Dg(y_s + \eta y_{s,t})\Big] R^y_{s,t}\Big\}d\eta \notag\\
		&&+ \int_0^1 \Big[Dg(\ty_s + \eta \ty_{s,t}) - Dg(\ty_s)\Big] \Big[g(\ty_s)-g(y_s)\Big] \otimes x_{s,t}d\eta \notag\\
		&&+ \Big(\int_0^1 \int_0^1 D^2g(\ty_s + \mu\eta \ty_{s,t}) \eta (\ty_{s,t}-y_{s,t})d\mu d\eta \Big)g(y_s) \otimes x_{s,t} \\
		&&+ \Big(\int_0^1 \int_0^1\Big[D^2g(\ty_s + \mu\eta \ty_{s,t})- D^2g(y_s + \mu\eta y_{s,t})\eta y_{s,t}\Big] d\mu d\eta\Big) g(y_s) \otimes x_{s,t}.
	\end{eqnarray*} 
	This proves $[g(\ty) - g(y)]^\prime_s = Dg(\ty_s)g(\ty_s) - Dg(y_s)g(y_s)$ which has the form $Q(\ty_s) - Q(y_s)$. Notice that $\|Q(\ty_s) - Q(y_s)\|\leq 2C_g^2 \|z_s\|$ and
	\begin{eqnarray*}
		\ltn Q(\ty) - Q(y) \rtn_{\tp,[s,t]} &\leq& C_Q\Big(\ltn z \rtn_{\tp,[s,t]} + \|z\|_{\infty,[s,t]} \ltn y \rtn_{\tp,[s,t]}\Big)\\
		&\leq& 2C_g^2 \Big(\ltn z \rtn_{\tp,[s,t]} + \|z\|_{\infty,[s,t]} \ltn y \rtn_{\tp,[s,t]}\Big).
	\end{eqnarray*} 	
	In addition
	\begin{eqnarray*}
	&&	\ltn R^{g(\ty)-g(y)}\rtn_{\tq,[s,t]^2}\\
	 &\leq& C_g \ltn R^{z} \rtn_{\tq,[s,t]^2} + C_g \|z\|_{\infty,[s,t]} \ltn R^{y} \rtn_{\tq,[s,t]^2} \\
		&&+ \frac{1}{2}C_g^2 \ltn x \rtn_{\tp,[s,t]} \Big[\ltn z \rtn_{\tp,[s,t]} + \|z\|_{\infty} \Big(\ltn \ty \rtn_{\tp,[s,t]}+\ltn y \rtn_{\tp,[s,t]}\Big)\Big].
	\end{eqnarray*}
	Using the fact that
	\begin{eqnarray}\label{roughest1}
	\|z_{s,t}\| &\leq& \Big\| \int_s^t [g(\ty_u)-g(y_u)]dx_u\Big\| \notag\\
	&\leq& C_g \|z_s\| \ltn x \rtn_{\tp,[s,t]} + 2C^2_g \|z_s\| \ltn \X \rtn_{\tq,[s,t]^2} \notag\\
	&&+ C_p \Big\{\ltn x \rtn_{\tp,[s,t]} \ltn R^{g(\ty)-g(y)} \rtn_{\tq,[s,t]^2} \notag\\
	&& + \ltn \X \rtn_{\tq,[s,t]^2} \ltn [g(\ty)-g(y)]^\prime\rtn_{\tp,[s,t]} \Big \},
	\end{eqnarray}
	we can now estimate
	\begin{eqnarray}\label{roughest2}
	\ltn z \rtn_{\tp,[s,t]} &\leq& 2\Big\{C_g \ltn \bx \rtn_{\tp,[s,t]} \vee  C_g^2 \ltn \bx \rtn_{\tp,[s,t]}^2  \Big\} \times \notag\\
	&&\times  \Big\{\|z\|_{\infty,[s,t]} \Big[1+ C_p \Big(\ltn \ty,R^{\ty} \rtn_{\tp,[s,t]} + \ltn y,R^y \rtn_{\tp,[s,t]}\Big)\Big] \notag\\
	&&+ C_p \ltn z,R^z \rtn_{\tp,[s,t]} \Big\} \notag \\
	&\leq&  2C_p\Big\{C_g \ltn \bx \rtn_{\tp,[s,t]} \vee  C_g^2 \ltn \bx \rtn_{\tp,[s,t]}^2  \Big\} \times \notag \\
	&&\times  \Big(1 +\ltn \ty,R^{\ty} \rtn_{\tp,[s,t]} + \ltn y,R^y \rtn_{\tp,[s,t]} \Big)\times \notag\\
	&& \times \Big(\|z_s\|+ \ltn z,R^z \rtn_{\tp,[s,t]} \Big).
	\end{eqnarray}
	The similar estimate for $\ltn R^z \rtn_{\tq,[s,t]}$ is already included in the estimate \eqref{roughest2}. By combining \eqref{roughest2} with \eqref{yRpvarest}, we obtain
	\begin{eqnarray}\label{roughest3}
	&&\ltn z,R^z \rtn_{\tp,[s,t]} \notag\\
	&\leq& 4C_p\Big\{C_g \ltn \bx \rtn_{\tp,[s,t]} \vee  C_g^2 \ltn \bx \rtn_{\tp,[s,t]}^2  \Big\} \times \notag\\
	&&\times  \Big(1 +\ltn \ty,R^{\ty} \rtn_{\tp,[s,t]} + \ltn y,R^y \rtn_{\tp,[s,t]} \Big) \Big(\|z_s\|+ \ltn z,R^z \rtn_{\tp,[s,t]} \Big) \notag\\
	&\leq& 4C_p\Big\{C_g \ltn \bx \rtn_{\tp,[s,t]} \vee  C_g^2 \ltn \bx \rtn_{\tp,[s,t]}^2  \Big\} \Big(1 +\frac{2}{C_p}N^{\frac{2p-1}{p}}_{\frac{1}{8C_pC_g},[a,b]}(\bx) \Big) \times\notag\\
	&&\times \Big(\|z_s\|+ \ltn z,R^z \rtn_{\tp,[a,b]} \Big), \notag\\
	&\leq& 4C_p\Big\{C_g \ltn \bx \rtn_{\tp,[s,t]} \vee  C_g^2 \ltn \bx \rtn_{\tp,[s,t]}^2  \Big\} \times \notag\\
	&&\times \Big\{1 + \frac{1}{C_p} 2^{\frac{2p-1}{p}} \Big(1+ [8C_pC_g]^{2p-1} \ltn \bx \rtn_{\tp,[a,b]}^{2p-1}\Big)\Big\} \times \notag\\
	&& \times \Big(\|z_s\|+ \ltn z,R^z \rtn_{\tp,[a,b]} \Big), 
	\end{eqnarray}
	which yields 
	\[
	\ltn z,R^z \rtn_{\tp,[s,t]} \leq \|z_s\|\ \text{whenever}\ 8C_pC_g C(\bx,[a,b])\ltn \bx \rtn_{\tp,[s,t]}   \leq 1.
	\]
	Therefore, \eqref{ydiff} is followed directly from the usage of the greedy sequence of stopping times \eqref{greedybarN}, which is similar to \eqref{ysupest} and \eqref{yRpvarest}. In particular $\bar{N}$ can be estimated by
	\begin{eqnarray}\label{Nbarest}
	\bar{N}_{[a,b]}(\bx) &\leq& 1 + \Big(8C_pC_g C(\bx,[a,b])\ltn \bx \rtn_{\tp,[a,b]}\Big)^p\notag\\
	&\leq& 1 + \Big(8C_pC_g \ltn \bx \rtn_{\tp,[a,b]}\Big)^p \times \notag\\
	&& \times \Big\{1 + \frac{1}{C_p} 2^{\frac{2p-1}{p}} \Big(1+ [8C_pC_g]^{2p-1} \ltn \bx \rtn_{\tp,[a,b]}^{2p-1}\Big)\Big\}^p.
	\end{eqnarray}
\end{proof} 	

\begin{theorem}\label{prop2}
	The solution $y_t(\bx,y_a)$ of \eqref{RDEpart} is differentiable w.r.t. initial condition $y_a$, moreover, its derivatives $\frac{\partial y_t}{\partial y_a} (\bx,y_a)$ is the matrix solution of the linearized rough differential equation
	\begin{equation}\label{RDEpartlin}
	d\xi_t = Dg(y_t)\xi_t dx_t
	\end{equation}
\end{theorem}
\begin{proof}
	The proof is divided into several steps.
	
	{\bf Step 1:} First, for a fixed solution $y_t(\bx,y_a)$ on $[a,b]$, we need to prove the existence and uniqueness of the solution of the linearized rough differential equation \eqref{RDEpartlin}, which has the time dependent coefficient $\Sigma_t := Dg(y_t)$. To do that, we simply follow Gubinelli's method by considering the solution mapping
	$H_t = \xi_a + \int_a^t \Sigma_s \xi_s dx_s$ on the set $\cD^{2\alpha}_x([a,b],\xi_a,\Sigma_a \xi_a)$ of the controlled paths $\xi_t$  with the fixed initial conditions $\xi_a, \xi^\prime_a = \Sigma_a \xi_a$. Note that $\Sigma_t = Dg(y_t)$ is also controlled by $x$ with
	\[
	\Sigma_{s,t} = \int_0^1 D^2g(y_s + \eta y_{s,t})(g(y_s) \otimes x_{s,t} + R^y_{s,t}) d\eta,
	\]
	thus 
	\[
	\Sigma^\prime_s = D^2_g(y_s)g(y_s),\quad \|R^\Sigma_{s,t}\| \leq C_g \|\R^y_{s,t}\| + \frac{1}{2}C_g^2 \|y_{s,t}\| \|x_{s,t}\|.
	\]
	As a result, $\Sigma_t \xi_t$ is also controlled by $x$ with $[\Sigma_\cdot \xi_\cdot]_s^\prime = \Sigma_s^\prime \xi_s + \Sigma_s \xi^\prime_s$ and 
	\[
	\|R^{\Sigma_\cdot \xi_\cdot}_{s,t}\| \leq \|\Sigma_{s,t}\| \|\xi_{s,t}\| + \|\xi_s\| \|R^{\Sigma}_{s,t}\| + \|\Sigma_s\|\|R^\xi_{s,t}\|.
	\]
	It then enable us to estimate 
	\begin{eqnarray*}
		&&\|H_{s,t} - \Sigma_s \xi_s \otimes x_{s,t} + [\Sigma_s^\prime \xi_s + \Sigma_s \xi^\prime_s]  \X_{s,t} \| \\
		&\leq& \Big\|\int_s^t \Sigma_u \xi_u dx_u - \Sigma_s \xi_s  \otimes x_{s,t} + [\Sigma_s^\prime \xi_s + \Sigma_s \xi^\prime_s]   \X_{s,t}\Big\| \\
		&&+ C_\alpha (t-s)^{3\alpha} \Big(\ltn x \rtn_{\alpha,[s,t]} \ltn R^{\Sigma_\cdot \xi_\cdot}\rtn_{2\alpha,[s,t]} + \ltn \X \rtn_{2\alpha,[s,t]} \ltn [\Sigma_\cdot \xi_\cdot]^\prime \rtn_{\alpha,[s,t]} \Big)
	\end{eqnarray*}
	where
	\begin{eqnarray*}
		\ltn [\Sigma_\cdot \xi_\cdot]^\prime \rtn_{\alpha} &\leq& \|\Sigma^\prime\|_\infty \ltn \xi \rtn_\alpha + \ltn \Sigma^\prime \rtn_\alpha \|\xi\|_\infty + \|\Sigma\|_\infty \ltn \xi^\prime \rtn_\alpha + \ltn \Sigma \rtn_\alpha \|\xi^\prime \|_\infty,\\
		\ltn R^{\Sigma_\cdot \xi_\cdot} \rtn_{2\alpha} &\leq& \ltn \Sigma \rtn_\alpha \ltn \xi \rtn_\alpha + \ltn R^\Sigma \rtn_{2\alpha} \|\xi\|_\infty + \|\Sigma\|_\infty \ltn R^\xi \rtn_{2\alpha}\quad \text{with}\\ 
		\|\xi^\prime \|_\infty &\leq& \|\xi^\prime_a\|+ (t-a)^\alpha \ltn \xi^\prime \rtn_\alpha \leq \|\Sigma\|_\infty \|\xi_a\| + (t-a)^\alpha \ltn \xi,\xi^\prime \rtn_{2\alpha}, \\
		\ltn \xi\rtn_\alpha &\leq& \|\xi^\prime \|_\infty \ltn x \rtn_\alpha + (t-a)^\alpha \ltn R^\xi\rtn_{2\alpha} \\
		&\leq& \|\Sigma\|_\infty \ltn x \rtn_\alpha \|\xi_a\| + \Big(\ltn x \rtn_\alpha \vee 1 \Big) (t-a)^\alpha \ltn \xi, \xi^\prime \rtn_{2\alpha},\\
		\|\xi\|_\infty &\leq& \|\xi_a\| + (t-a)^\alpha \ltn \xi\rtn_\alpha \\
		&\leq& \|\xi_a\|\Big(1+ \|\Sigma\|_\infty (t-a)^\alpha \ltn x \rtn_\alpha\Big)\\
		&&+ (t-a)^{2\alpha} \Big(\ltn x \rtn_\alpha \vee 1\Big) \ltn \xi, \xi^\prime \rtn_{2\alpha}.
	\end{eqnarray*}
	A direct computation then shows that
	\begin{eqnarray*}
		\ltn [\Sigma_\cdot \xi_\cdot]^\prime \rtn_{\alpha} &\leq& \Big\{ \|\Sigma \|_\infty + (t-a)^\alpha \ltn \Sigma \rtn_\alpha  + \|\Sigma^\prime \|_\infty (t-a)^\alpha (\ltn x \rtn_\alpha \vee 1)\\
		&& + \ltn \Sigma^\prime \rtn_\alpha (t-a)^{2\alpha}(\ltn x \rtn_\alpha \vee 1) \Big\} \ltn \xi, \xi^\prime \rtn_{2\alpha}\\
		&&+ \Big\{\ltn \Sigma\rtn_\alpha \|\Sigma\|_\infty + \|\Sigma^\prime \|_\infty \|\Sigma\|_\infty \ltn x \rtn_\alpha\\
		&& + \ltn \Sigma^\prime \rtn_\alpha \Big(1 + \|\Sigma\|_\infty (t-a)^\alpha \ltn x \rtn_\alpha\Big) \Big\}\|\xi_a\|\\
		\ltn R^{\Sigma_\cdot \xi_\cdot} \rtn_{2\alpha} &\leq& \Big\{\|\Sigma\|_\infty + \ltn \Sigma \rtn_\alpha (t-a)^\alpha (\ltn x \rtn_\alpha \vee 1)\\
		&& + \ltn R^\Sigma \rtn_{2\alpha} (t-a)^{2\alpha} (\ltn x \rtn_\alpha \vee 1) \Big \} \ltn \xi, \xi^\prime \rtn_{2\alpha} \\
		&&+ \Big\{\ltn \Sigma \rtn_\alpha \|\Sigma \|_\infty \ltn x \rtn_\alpha \\
		&&+ \ltn R^{\Sigma} \rtn_{2\alpha} \Big(1 + \|\Sigma\|_\infty (t-a)^\alpha \ltn x \rtn _\alpha\Big) \Big \} \|\xi_a\|.
	\end{eqnarray*}
	Hence 
	\begin{eqnarray*}
		\ltn R^H \rtn_{2\alpha} &\leq& \Big(\|\Sigma^\prime \|_\infty \|\xi\|_\infty + \|\Sigma\|_\infty \|\xi^\prime \|_\infty \Big) \ltn \X \rtn_{2\alpha}\\
		&& + C_\alpha (t-a)^\alpha \Big(\ltn x \rtn_\alpha \ltn R^{\Sigma_\cdot \xi_\cdot} \rtn_{2\alpha} + \ltn \X \rtn_{2\alpha} \ltn [\Sigma_\cdot \xi_\cdot]^\prime \rtn_\alpha \Big);\\
		\ltn H^\prime \rtn_\alpha &\leq&  \|\Sigma\|_\infty \ltn \xi \rtn_\alpha + \ltn \Sigma \rtn_\alpha \|\xi\|_\infty \\
		&\leq& \Big\{\|\Sigma\|_\infty (t-a)^\alpha (\ltn x \rtn_\alpha \vee 1) + \ltn \Sigma \rtn_\alpha (t-a)^{2\alpha} (\ltn x \rtn_\alpha \vee 1) \Big \} \ltn \xi, \xi^\prime \rtn_{2\alpha}\\
		&&+ \Big\{\|\Sigma \|_\infty^2 \ltn x \rtn_\alpha + \ltn \Sigma \rtn_\alpha \Big[1 + \|\Sigma\|_\infty (t-a)^\alpha \ltn x \rtn_\alpha\Big]  \Big \}\|\xi_a\|.
	\end{eqnarray*}
	All things combined, we have just showed that there exists constants 
	\[
	M_1=M_1(\Sigma,[a,b],x,\X), M_2 = M_2(\Sigma,[a,b],x,\X)
	\]
	such that
	\begin{eqnarray}\label{normest}
	\ltn H^\prime \rtn_\alpha + \ltn R^H \rtn_{2\alpha} &\leq& M_2 \|\xi_a\| \notag\\
	&& + M_1 \Big[(t-a)^\alpha + \ltn x \rtn_\alpha + \ltn \X \rtn_{2\alpha} \Big] \ltn (\xi, \xi^\prime) \rtn_{2\alpha}.
	\end{eqnarray}
	This implies that on every interval $[\bar{\tau}_k,\bar{\tau}_{k+1}]$ of the greedy sequence of stopping times~  $\{\bar{\tau}_i(\frac{1}{2M_1},I,\alpha)\}_{i \in \N}$ as in \eqref{greedytimeHol}, the solution mapping is a contraction from the set 
	\[
	\Big\{\cD^{2\alpha}_x([a,b],\xi_a,\Sigma_a \xi_a): \ltn (\xi,\xi^\prime)\rtn_{2\alpha,[\bar{\tau}_k,\bar{\tau}_{k+1}]} \leq 2M_2 \|\xi_{\bar{\tau}_k}\| \Big\}  
	\]
	into itself, hence there exists a solution of \eqref{RDEpartlin} on every interval $[\bar{\tau}_k,\bar{\tau}_{k+1}]$. Because of the linearity, it follows from the same estimate \eqref{normest} with $\|\xi_a- \bar{\xi}_a\| = 0$ that $\ltn (\xi - \bar{\xi}, \xi^\prime - \bar{\xi}^\prime) \rtn_{2\alpha} = 0$, which proves the uniqueness. Finally, the concatenation of solutions on intervals $[\bar{\tau}_k,\bar{\tau}_{k+1}]$ then proves the existence and uniqueness of the solution of \eqref{RDEpartlin} on $[a,b]$.
	
	{\bf Step 2:} Denote by $\Phi(t,x,y_a)$ the solution matrix of the linearized system \eqref{RDEpartlin}, then $\xi_t = \Phi(t,x,y_a)(\bar{y}_a-y_a)$ is the solution of \eqref{RDEpartlin} given initial point $\xi_a =\bar{y}_a-y_a$. Assign $r_t := \bar{y}_t - y_t - \xi_t$, then $r_a =0$ and
	\begin{eqnarray}\label{equationrt}
	r_t &=& \int_a^t \Big[\int_0^1 Dg(y_s + \eta(\bar{y}_s-y_s)) - Dg(y_s)\Big](\bar{y}_s -y_s)d\eta dx_s \notag\\
	&& + \int_a^t Dg(y_s) r_s dx_s,\notag\\
	&=& e_{a,t}+\int_a^t Dg(y_s) r_s dx_s, \quad \forall t \in [a,b],
	\end{eqnarray}	
	where 
	\[
	e_{a,t} = \int_a^t \int_0^1 \Big[Dg(y_s + \eta(\bar{y}_s-y_s)) - Dg(y_s)\Big](\bar{y}_s - y_s)d\eta dx_s 
	\]
	and $e$ is also controlled by $x$ with $e_{a,a} =0$. We are going to estimate 
	$\|r\|_{\infty,[a,b]}$ and $\ltn r,R^r\rtn_{\tp,[a,b]}$ through $\|e^\prime\|_{\infty,[a,b]},  \ltn e^\prime \rtn_{\tp,[a,b]}, \ltn R^e \rtn_{\tq,[a,b]^2}$. First observe that 
	\[
	r_{s,t} = e_{s,t} + \int_s^t Dg(y_u)r_u dx_u = e^\prime  \otimes x_{s,t} + R^e_{s,t} + \int_s^t Dg(y_u)r_u dx_u,
	\]
	which yields $r^\prime_s = e^\prime_s + Dg(y_s)r_s$ and
	\begin{eqnarray}\label{Rr}
	\|R^r_{s,t}\| &\leq& \|R^e_{s,t}\| + \|[Dg(y)r]^\prime_s \| \|\X_{s,t}\|\notag\\
	&& + C_p \Big(\ltn x \rtn_{p} \ltn R^{Dg(y)r}\rtn_{q} + \ltn \X \rtn_{q} \ltn [Dg(y)r]^\prime \rtn_{p} \Big).
	\end{eqnarray}
	A direct computation shows that
	\begin{eqnarray*}
		&&\Big\| Dg(y_t)r_t - Dg(y_s)r_s - \Big [ D^2_g (y_s) g(y_s) r_s + Dg(y_s)r^\prime_s \Big] \otimes x_{s,t} \Big\| \\
		&\leq& \Big\|\int_0^1 D^2g(y_s + \eta y_{s,t})R^y_{s,t} r_s d\eta \Big\| + \frac{1}{2}C_g^2 \|y_{s,t}\|\|r_s\| \|x_{s,t}\| \\
		&&+ \|Dg(y_s)R^y_{s,t}\| + C_g \|y_{s,t}\| \|r_{s,t}\|, 
	\end{eqnarray*}
	which implies that $[Dg(y)r]^\prime_s = D^2g(y_s)[g(y_s),r_s] + Dg(y_s)r^\prime_s$ and
	\begin{eqnarray*}
		\|R^{D(y)r}_{s,t}\| &\leq& C_g \|R^y_{s,t}\| \|r\|_\infty + \frac{1}{2}C_g^2 \|r\|_\infty \|y_{s,t}\| \|x_{s,t}\| \\
		&& + C_g \|R^r_{s,t}\| + C_g \|y_{s,t}\| \|r_{s,t}\| \\
		\Rightarrow \ltn R^{Dg(y)r}\rtn_{q} &\leq& C_g \|r\|_\infty \ltn R^y \rtn_q + \frac{1}{2}C_g^2 \|r\|_\infty \ltn y \rtn_p \ltn x \rtn_p \\
		&& + C_g \ltn R^r \rtn_q + C_g \ltn y \rtn_p \ltn r \rtn_p.
	\end{eqnarray*}
	Similarly, we can show that
	\begin{eqnarray*}
		\|[Dg(y)r]^\prime\|_\infty &\leq& C^2_g \|r\|_\infty + C_g \|r^\prime\|_\infty \\
		\ltn [Dg(y)r]^\prime \rtn_p &\leq& 2C_g^2 \|r\|_\infty \ltn y \rtn_p + C_g^2 \ltn r \rtn_p + C_g \ltn r^\prime \rtn_p + C_g \|r^\prime\|_\infty \ltn y \rtn_p.
	\end{eqnarray*}
	Combining all the above estimates into \eqref{Rr}, we obtain
	\begin{eqnarray*}
		&&\ltn R^r \rtn_q \leq C_pC_g \ltn x \rtn_p \ltn R^r \rtn_q  \\
		&&+ \ltn r \rtn_p \Big\{2C_g^2 \ltn \X \rtn_q + C_p \ltn x \rtn_p \Big[C_g \ltn R^y \rtn_q + \frac{1}{2}C_g^2 \ltn x \rtn_p \ltn y \rtn_p + C_g \ltn y \rtn_p\Big] \\
		&& + C_g \ltn \X \rtn_q (4C_g^2 \ltn y \rtn_p + 2C_g^2) \Big\}\\
		&&+ \|r_s\| \Big\{2C_g^2 \ltn \X \rtn_q + C_p \ltn x \rtn_p \Big[C_g \ltn R^y \rtn_q + \frac{1}{2}C_g^2 \ltn x \rtn_p \ltn y \rtn_p\Big] + 4C_pC_g^2 \ltn \X \rtn_q \Big\}\\
		&&+ \ltn R^e \rtn_q + C_g \ltn \X \rtn_q\|e^\prime\|_\infty + C_p \ltn \X \rtn_q (C_g \ltn e^\prime \rtn_p + C_g \ltn y \rtn_p \|e^\prime\|_\infty) =: \bar{R} ,
	\end{eqnarray*}
	and similarly 
	\[
	\ltn r \rtn_p \leq \|e^\prime \|_\infty \ltn x \rtn_p + C_g \|r\|_\infty \ltn x \rtn_p + \bar{R} \leq  \|e^\prime \|_\infty \ltn x \rtn_p + C_g \ltn x \rtn_p (\|r_s\| + \ltn r \rtn_p)+ \bar{R}. 
	\] 	
	Therefore, taking into account \eqref{yRpvarest} we have just proved that there exists a constant 
	\[
	M = M(p,[a,b], \ltn \bx \rtn_{\tp,[a,b]})>1 
	\]
	such that
	\begin{eqnarray*}
	&&	\ltn r,R^r\rtn_{\tp,[s,t]} \\
	&\leq& M \Big(C_g \ltn x \rtn_{\tp,[s,t]} + C_g^2 \ltn \X \rtn_{\tq,[s,t]^2} \Big) \Big(\|r_s\| + \ltn r,R^r \rtn_{\tp,[s,t]} \Big) \\
		&& + M \Big(\|e^\prime \|_{\infty,[a,b}] + \ltn e^\prime \rtn_{\tp,[a,b]} + \ltn R^e \rtn_{\tq,[a,b]^2} \Big)\\
		&\leq& 2M \Big(C_g \ltn \bx \rtn_{\tp,[s,t]} \vee C_g^2 \ltn \bx \rtn^2_{\tq,[s,t]} \Big) \Big(\|r_s\| + \ltn r,R^r \rtn_{\tp,[s,t]} \Big) \\
		&& + M \Big(\|e^\prime \|_{\infty,[a,b}] + \ltn e^\prime \rtn_{\tp,[a,b]} + \ltn R^e \rtn_{\tq,[a,b]^2} \Big).
	\end{eqnarray*}
	This implies
	\[
	\ltn r,R^r\rtn_{\tp,[s,t]}  \leq \|r_s\| + 2M \Big(\|e^\prime \|_{\infty,[a,b}] + \ltn e^\prime \rtn_{\tp,[a,b]} + \ltn R^e \rtn_{\tq,[a,b]^2} \Big)
	\]
	whenever $2M C_g \ltn \bx \rtn_{\tp,[s,t]}  \leq \frac{1}{2}$. The similar estimates to \eqref{ysupest} and \eqref{yRpvarest}, with the usage of the greedy sequence of stopping times $\{\tau_k(\frac{1}{4MC_g})\}_{k \in \N}$, lead to
	\begin{eqnarray}\label{rEst}
	&& \|r\|_{\infty,[a,b]} \vee \ltn r,R^r\rtn_{\tp,[a,b]} \notag \\
	&\leq& N_{\frac{1}{4MC_g},[a,b]}(\bx)^{\frac{p-1}{p}}  e^{(\log 2) N_{\frac{1}{4MC_g},[a,b]}(\bx)} \notag \\
	&& \times \Big\{\|r_a\| + 2M \Big(\|e^\prime \|_{\infty,[a,b}] + \ltn e^\prime \rtn_{\tp,[a,b]} + \ltn R^e \rtn_{\tq,[a,b]^2} \Big) \Big\}\notag \\
	&\leq& 2M e^{(1+ \log 2) N_{\frac{1}{4MC_g},[a,b]}(\bx)}\times \notag\\
	&& \times \Big(\|e^\prime \|_{\infty,[a,b}] + \ltn e^\prime \rtn_{\tp,[a,b]} + \ltn R^e \rtn_{\tq,[a,b]^2} \Big) ,
	\end{eqnarray}
	where we use the fact that $r_a =0$.
	
	{\bf Step 3:} Assign $z_s = \ty_s-y_s$, it follows from \eqref{equationrt} that 
	\begin{eqnarray*}
	e^\prime_s &=& \int_0^1 [Dg(y_s + \eta z_s) - Dg(y_s)] z_s d\eta \\
	&=& \int_0^1 \int_0^1 D^2g\Big((1-\mu\eta)y_s + \mu \eta \ty_s\Big) [z_s,z_s] \eta d\mu d\eta = \bar{e}_s, 
	\end{eqnarray*}
	which 	is also controlled by $x$. As a result, $\|e^\prime_s\| \leq \frac{1}{2}C_g \|z_s\|^2 $ and 
	\begin{eqnarray}\label{eprime}
	\|e^\prime \|_{\infty,[a,b]} &\leq& \frac{1}{2}C_g \|z\|^2_{\infty,[a,b]},\\
	\ltn e^\prime \rtn_{\tp,[a,b]} &\leq& C_g \Big(\ltn y \rtn_{\tp,[a,b]} \vee \ltn \ty \rtn_{\tp,[a,b]}\Big) \|z\|^2_{\infty,[a,b]} \notag\\
	&&+ 2C_g \|z\|_{\infty,[a,b]} \ltn z \rtn_{\tp,[a,b]}.\notag
	\end{eqnarray}
	On the other hand,
	\begin{eqnarray*}
	\|R^e_{s,t}\| &\leq& \|\bar{e}^\prime \|_{\infty} \|\X_{s,t}\| + C_p \Big( \ltn x \rtn_{\tp,[a,b]}\ltn R^{\bar{e}} \rtn_{\tq,[a,b]^2} \notag\\
	&& + \ltn \X \rtn_{\tq,[a,b]} \ltn \bar{e}^\prime \rtn_{\tp,[a,b]} \Big) \notag\\
	\end{eqnarray*}
	thus	
	\begin{eqnarray}\label{Re}		
	\ltn R^e \rtn_{\tq,[a,b]^2} &\leq&  \|\bar{e}^\prime \|_{\infty,[a,b]} \ltn \X\rtn_{\tq,[a,b]^2} + C_p \Big( \ltn x \rtn_{\tp,[a,b]}\ltn R^{\bar{e}} \rtn_{\tq,[a,b]^2} \notag\\
	&& + \ltn \X \rtn_{\tq,[a,b]^2} \ltn \bar{e}^\prime \rtn_{\tp,[a,b]} \Big),
	\end{eqnarray}
	where a direct computation shows that
	\begin{eqnarray*}
		\bar{e}^\prime_s &=& \int_0^1 \int_0^1 \Big\{ D^3g\Big((1-\mu\eta)y_s + \mu \eta \ty_s\Big) [(1-\mu\eta)y^\prime_s + \mu \eta \ty^\prime_s,z_s,z_s]   \\
		&&+ 2 D^2g\Big((1-\mu\eta)y_s + \mu \eta \ty_s\Big) [z^\prime_s,z_s]\Big\} \eta d\mu d\eta,\\
		R^{\bar{e}}_{s,t} &=&\int_0^1 \int_0^1 \Big\{ R^{D^2g(\ty,y)}_{s,t}[z_t,z_t] \\
		&&+ D^2g(\ty,y)\Big([z_s,R^z_{s,t}] + [z^\prime_s x_{s,t}, z^\prime_s x_{s,t}+ R^z_{s,t}] + [R^z_{s,t},z_t]\Big) \\
		&&+ [D^2g(\ty,y)]^\prime_s x_{s,t} \Big([z_s + z_t, z^\prime_s + R^z_{s,t}]\Big)\Big\} \eta d\mu d\eta.
	\end{eqnarray*}
	We therefore can show that there exists a generic constant 
	\[
	D = D(p,[a,b], \ltn \bx \rtn_{\tp,[a,b]}) 
	\]
	such that
	\begin{eqnarray}\label{barEest}
	&&\|\bar{e}^\prime\|_{\infty,[a,b]} \vee \ltn \bar{e}^\prime \rtn_{\tp,[a,b]} \vee \ltn R^{\bar{e}} \rtn_{\tq,[a,b]^2} \notag\\
	&\leq& D \Big(\|z\|_{\infty,[a,b]} + \ltn z,R^z\rtn_{\tp,[a,b]}+\|z^\prime\|_{\infty,[a,b]} + \ltn z^\prime\rtn_{\tp,[a,b]}\Big)^2. 
	\end{eqnarray}
	By replacing \eqref{eprime}, \eqref{Re}, \eqref{barEest} into \eqref{rEst}, and using \eqref{ydiff}, we conclude that there exists a generic constant $D = D(p,[a,b], \ltn \bx \rtn_{\tp,[a,b]})$ such that 
	\begin{equation}\label{diffest}
	\|\ty_\cdot(\bx,\ty_a) - y_\cdot(\bx,y_a) - \xi_\cdot(\bx,\ty_a-y_a)\|_{\infty,[a,b]}  \leq D \|\ty_a - y_a\|^2.
	\end{equation}
	This, combined with the linearity of $\xi$ w.r.t. $\ty_a-y_a$, shows the differentiability of $y_t(\bx,y_a)$ w.r.t. $y_a$.
	
\end{proof}	

By similar arguments in the backward direction, one can prove the following.
\begin{corollary}\label{diffcor}
	Consider the backward rough differential equation
	\begin{equation}\label{backwardRDE}
	h_b = h_t + \int_t^b g(h_u) dx_u,\quad \forall t \in [a,b].
	\end{equation}	
	Then the solution $h_t(\bx,h_b)$ of \eqref{backwardRDE} is differentiable w.r.t. initial condition $h_b$, moreover, its derivatives $\frac{\partial h_t}{\partial h_b} (\bx,h_b)$ is the matrix solution of the linearized backward rough differential equation
	\begin{equation}
	\eta_b = \eta_t + \int_t^b Dg(h_u) \eta_u dx_u,\quad \forall t \in [a,b].
	\end{equation}	
	Moreover, the estimates \eqref{ydiff} also hold for the solution $h_t$ of the backward equation \eqref{backwardRDE}.
\end{corollary}

\begin{corollary}\label{lipschitzcor}
	Denote by $\varphi(t,\bx,y_a)$ the solution mapping of the rough equation \eqref{RDEpart}. Then $ \frac{\partial \varphi}{\partial y}(t,\bx,y_a)$ is globally Lipschitz continuous w.r.t. $y_a$.
\end{corollary}

\begin{proof}
	Observe from \eqref{diffest} and \eqref{ydiff} that
	\begin{eqnarray*}
	\|\xi_\cdot(\bx,\ty_a-y_a)\|_{\infty,[a,b]} &\leq& \|\ty_\cdot(\bx,\ty_a) - y_\cdot(\bx,y_a) \| + D \|\ty_a - y_a\|^2 \\
	&\leq& \|\ty_a -y_a\|e^{(\log 2) \bar{N}_{[a,b]}(\bx)}+ D \|\ty_a - y_a\|^2, 
	\end{eqnarray*}
	thus by fixing $y_a$, then dividing both sides by $\|\ty_a - y_a\|$ and letting $ \|\ty_a - y_a\|$ to zero, we obtain
	\begin{eqnarray}\label{phiest}
	\|\frac{\partial \varphi}{\partial y}(t,\bx,y_a)\|_{\infty,[a,b]} &\leq& e^{(\log 2) \bar{N}_{[a,b]}(\bx)},\notag\\
	\| \frac{\partial \varphi}{\partial y}(t,\bx,y_a) \|_{\tp,[a,b]} &\leq&  \bar{N}^{\frac{p-1}{p}}_{[a,b]}(\bx)e^{(\log 2) \bar{N}_{[a,b]}(\bx)}.
	\end{eqnarray}
	Similarly if we fix $\ty_a$ then
	\begin{eqnarray}\label{phiestya2}
	\|\frac{\partial \varphi}{\partial y}(t,\bx,\ty_a)\|_{\infty,[a,b]} &\leq& e^{(\log 2) \bar{N}_{[a,b]}(\bx)},\notag\\
	\| \frac{\partial \varphi}{\partial y}(t,\bx,\ty_a) \|_{\tp,[a,b]} &\leq& \bar{N}^{\frac{p-1}{p}}_{[a,b]}(\bx)e^{(\log 2) \bar{N}_{[a,b]}(\bx)}.
	\end{eqnarray}
	Next for $\xi_t = \frac{\partial \varphi}{\partial y}(t,\bx,y_a) \xi_a$ and $\bar{\xi}_t = \frac{\partial \varphi}{\partial y}(t,\bx,\ty_a) \xi_a$, we consider the difference $r_t = \bar{\xi}_t - \xi_t$, which satisfies $r_a =0$ and
	\begin{eqnarray}\label{eqz1}
	r_{s,t} &=& \int_s^t \Big(Dg(\ty_u) \bar{\xi}_u - Dg(y_u) \xi_u\Big) dx_u\notag\\
	& = &\int_s^t \Big(Dg(\ty_u)  - Dg(y_u) \Big) \bar{\xi}_u dx_u +  \int_s^tDg(y_u) r_u dx_u \notag\\
	&=& e_{s,t} + \int_s^tDg(y_u) r_u dx_u.
	\end{eqnarray}
	Since \eqref{eqz1} has the form of \eqref{equationrt}, by the same arguments as in Step 2 of Theorem \ref{prop2} we obtain a similar estimate to \eqref{rEst}, i.e. there exists a constant $ M = M(p,[a,b], \ltn \bx \rtn_{\tp,[a,b]})>1 $ such that
	\begin{eqnarray}\label{eqz2}
	&&\|r\|_{\infty} \vee \ltn r,R^r \rtn_{\tp,[a,b]} \notag\\
	&\leq& 2M \Big(\|e^\prime \|_{\infty,[a,b}] + \ltn e^\prime \rtn_{\tp,[a,b]} + \ltn R^e \rtn_{\tq,[a,b]^2} \Big)\times \notag\\ 
	&& \times e^{(1+ \log 2) N_{\frac{1}{4MC_g},[a,b]}(\bx)},
	\end{eqnarray}
	where we use the fact that $r_a =0$. To estimate the terms in the right hand side of \eqref{eqz2}, we derive from \eqref{eqz1} that
	\[
	e^\prime_s =   \Big(Dg(\ty_s)  - Dg(y_s) \Big) \bar{\xi}_s = \int_0^1 D^2g\Big((1-\eta)y_s + \eta \ty_s\Big) [z_s ,\bar{\xi}_s] d\eta.
	\]
	Using \eqref{phiestya2} and the same estimates as in Step 3 of Theorem \ref{prop2}, we can prove that there exists a generic constant  $ M = M(p,[a,b], \ltn \bx \rtn_{\tp,[a,b]})>1 $ such that
	\begin{equation}\label{eqz3}
	\|e^\prime \|_{\infty,[a,b}], \ltn e^\prime \rtn_{\tp,[a,b]} \leq M \|z\|_{\tp,[a,b]}.
	\end{equation}
	Similarly, there exists a generic constant $M$ that
	\begin{equation}\label{eqz4}
	\|R^e \|_{\tq,[a,b]^2} \leq M \Big(\|z_a\| +\ltn z,R^z\rtn_{\tp,[a,b]} \Big).
	\end{equation}
	Replacing \eqref{eqz3} and \eqref{eqz4} into \eqref{eqz2}, we conclude that there exists a generic constant $ M>1 $ such that
	\begin{equation}\label{eqz2}
	\|r\|_{\infty} \vee \ltn r,R^r \rtn_{\tp,[a,b]} \leq M \Big(\|z_a\| +\ltn z,R^z\rtn_{\tp,[a,b]} \Big) \leq M \|z_a\|,
	\end{equation}
	where the last inequality is due to \eqref{ydiff}. This proves the global Lipschitz continuity of $ \frac{\partial \varphi}{\partial y}(t,\bx,y_a)$ w.r.t. $y_a$.
	
\end{proof}

We are now able to formulate and prove the existence theorem.
\begin{theorem}[Existence]\label{RDEexist}
	There exists a solution to \eqref{fSDE0} on the interval $[a,b]$ with the initial value $y_a$. 
\end{theorem}

\begin{proof}
	The main idea is similar to \cite{riedelScheutzow}, which applies the Doss-Sussmann transformation \cite{Sus78} for the rough integral, thus it is enough to prove the existence of the solution for the transformed system. The proof is divided into several steps.
	
	{\bf Step 1}. 	Denote by $\varphi(t,\bx,y_a)$ the solution mapping of the rough equation \eqref{RDEpart}, and by $\psi(t,\bx,h_b)$ the solution mapping of the backward equation \eqref{backwardRDE}. Then $\varphi(t,\bx) \circ \psi(t,\bx) = I^{d \times d}$, and due to Theorem \ref{prop2} and Corollary \ref{diffcor}, $\varphi, \psi$ are continuously differentiable w.r.t. $y_a$ and $h_b$ respectively. In addition,
	\begin{equation}\label{identity}
	\frac{\partial \varphi}{\partial y}(t,\bx,y_a) \frac{\partial \psi}{\partial h}(t,\bx,y_b) = \frac{\partial \psi}{\partial h}(t,\bx,y_b) \frac{\partial \varphi}{\partial y}(t,\bx,y_a) = I^{d\times d}.
	\end{equation}
	Similar to \eqref{phiest} and \eqref{phiestya2}, we obtain 
	\begin{equation}\label{psiest}
	\Big\|\frac{\partial \psi}{\partial h}(t,\bx,y_b)\Big\|_{\infty,[a,b]} \leq e^{(\log 2) \bar{N}_{[a,b]}(\bx)}.
	\end{equation}
	Now consider the ordinary differential equation
	\begin{equation}\label{RDEtransformed}
	\dot{z}_t = \frac{\partial \psi}{\partial h}(t,\bx,\varphi(t,x,z_t)) f(\varphi(t,x,z_t)) = F(t,z_t),\quad t \in [a,b].	
	\end{equation}
	Similar to Corollary \ref{lipschitzcor}, $ \frac{\partial \psi}{\partial h}(t,\bx,z)$ is globally Lipschitz continuous in $z$. On the other hand $\varphi(t,x,z)$ is also globally Lipschitz continuous w.r.t. $z$. Hence  $ \frac{\partial \psi}{\partial h}(t,\bx,\varphi(t,x,z_t)) $ is globally Lipschitz continuous w.r.t. $z$. Because $f$ is also globally Lipschitz continuous and of linear growth, it follows that $F(t,z)$ in the right hand side of \eqref{RDEtransformed} satisfies the local Lipschitz continuity and also the linear growth. Hence there exists a unique solution for \eqref{RDEtransformed}.
	
	{\bf Step 2}. Next consider the transformation $y_t = \varphi(t,\bx,z_t)$ for $t\in [a,b]$, we are going to prove that $y_t$ is a solution of the equation
	\begin{equation}\label{RDEgbounded}
	dy_t=f(y_t)dt + g(y_t)dx_t. 
	\end{equation}
	Indeed, it first follows that 
	\begin{eqnarray}\label{RDEeq1}
	y_{s,t} &=& \varphi(t,\bx,z_t) - \varphi(t,\bx,z_s) + \varphi(t,\bx,z_s) - \varphi(s,\bx,z_s) \notag\\
	&=& \Big[\varphi(t,\bx,z_t) - \varphi(t,\bx,z_s)\Big] + \int_s^t g(\varphi(u,\bx,z_s))d x_u.  
	\end{eqnarray}
	The term in the square bracket satisfies, given \eqref{ydiff} and Theorem \ref{prop2},
	\begin{eqnarray}\label{RDEeq2}
	&& \|\varphi(t,\bx,z_t) - \varphi(t,\bx,z_s) - \frac{\partial \varphi}{\partial z}(s,\bx,z_s) z_{s,t} \| \notag\\
	&\leq& \|\varphi(t,\bx,z_t) - \varphi(t,\bx,z_s) - \frac{\partial \varphi}{\partial z}(t,\bx,z_s) z_{s,t} \| \notag\\
	&&+ \Big\|  \frac{\partial \varphi}{\partial z}(t,\bx,z_s)  - \frac{\partial \varphi}{\partial z}(s,\bx,z_s) \Big \| \|z_{s,t} \| \notag\\
	&\leq& D\|z_{s,t}\|^2 + \Big\| \int_s^t Dg(\varphi(u,\bx,z_s)) \frac{\partial \varphi}{\partial z}(u,\bx,z_s) dx_u \Big \| \|z_{s,t} \| \notag\\
	&\leq& D (t-s)^{1+\alpha}
	\end{eqnarray}
	for some generic constant $D$. Also, due to \eqref{RDEtransformed} 
	\begin{eqnarray}\label{RDEeq3}
	\frac{\partial \varphi}{\partial z}(s,\bx,z_s) z_{s,t} 
	&=& \frac{\partial \varphi}{\partial z}(s,\bx,z_s)F(s,z_s) (t-s) + D(t-s)^2 \notag\\
	&=& \frac{\partial \varphi}{\partial z}(s,\bx,z_s)\frac{\partial \psi}{\partial h}(s,\bx,\varphi(s,x,z_s)) f(\varphi(s,x,z_s))(t-s) \notag\\
	&& + D(t-s)^2 \notag\\
	&=& f(y_s) (t-s) + D(t-s)^2.
	\end{eqnarray}
	On the other hand, it follows from \eqref{roughEst} that
	\begin{eqnarray}
	\int_s^t g(\varphi(u,\bx,z_s))d x_u &=& g(\varphi(s,\bx,z_s)) \otimes x_{s,t} \notag\\
	&& + Dg(\varphi(s,\bx,z_s)) [\varphi(\cdot,\bx,z_s)]^\prime_s  \X_{s,t} + D (t-s)^{3\alpha} \notag\\
	&=& g(y_s)  \otimes x_{s,t} + Dg(y_s) g(y_s)  \X_{s,t}+ D (t-s)^{3\alpha}.
	\end{eqnarray}
	By combining \eqref{RDEeq1}-\eqref{RDEeq3}, we now can write
	\begin{equation}\label{RDEeq4}
	y_{s,t} = f(y_s) (t-s) + g(y_s)  \otimes x_{s,t} + Dg(y_s) g(y_s)  \X_{s,t}+ \mathcal{O} ((t-s)^{3\alpha}).
	\end{equation}
	Equation \eqref{RDEeq4} implies that $y$ is controlled by the driving path $x$ with $y^\prime_s = g(\varphi(s,\bx,z_s)) = g(y_s)$ and so is $g(y)$ with $[g(y)]^\prime_s = Dg(y_s)y^\prime_s = Dg(y_s)g(y_s)$, for all $s\in [a,b]$. Now take any finite partition $\Pi$ of $[s,t]$ for $|\Pi| = \max_{[u,v]\in \Pi} |v-u| \ll 1$, it follows from \eqref{RDEeq4} that
	\begin{eqnarray}\label{RDEeq5}
	y_{s,t} &=& \sum_{[u,v]\in \Pi} y_{u,v} \notag\\
	&=& \sum_{[u,v]\in \Pi} f(y_u) (v-u) \notag\\
	&&+ \sum_{[u,v]\in \Pi} \Big( g(y_u)  \otimes x_{u,v} + Dg(y_u) g(y_u)  \X_{u,v} \Big)  +  \sum_{[u,v]\in \Pi} \mathcal{O}((v-u)^{3\alpha}) \notag\\
	&=& \sum_{[u,v]\in \Pi} f(y_u) (v-u) \notag\\
	&&+ \sum_{[u,v]\in \Pi} \Big( g(y_u)  \otimes x_{u,v} + Dg(y_u) g(y_u)  \X_{u,v} \Big) + \mathcal{O} (|\Pi|^{3\alpha-1}).
	\end{eqnarray}
	Let $|\Pi| \to 0$, the first term in \eqref{RDEeq5} converges to $\int_s^t f(y_u)du$ while the second term converges to the Gubinelli rough integral $\int_s^t g(y_u)dx_u$. We conclude that 
	\[
	y_{s,t} = \int_s^t f(y_u)du + \int_s^t g(y_u)dx_u,
	\]
	which proves the existence part.
	
\end{proof}
The next result provides the norm estimates for the solution of \eqref{fSDE0}.
\begin{theorem}\label{RDEnorm}
	The supremum and $p$ - variation norms of the solution are estimated as follows
	\begin{eqnarray}
	\|y\|_{\infty,[a,b]} &\leq&  \Big[\|y_a\| + \Big( \frac{\|f(0)\|}{C_f}+ \frac{1}{C_p} \Big)N_{[a,b]}(\bx)\Big] e^{4C_f (b-a)}, \label{estx} \\
	\ltn y,R^y \rtn_{p{\rm -var},[a,b]} &\leq& \Big[\|y_a\| + \Big(\frac{\|f(0)\|}{C_f} + \frac{1}{C_p}\Big) N_{[a,b]}(\bx)\Big]\times \notag\\
	&& \times e^{4C_f(b-a) }N^{\frac{p-1}{p}}_{[a,b]}(\bx) -\|y_a\|, \label{estx2}
	\end{eqnarray}
	where $\ltn y,R^y \rtn_{\tp,[s,t]} :=\ltn y \rtn_{\tp,[s,t]} + \ltn R^y \rtn_{\tq,[s,t]^2}$. 
\end{theorem}
\begin{proof}
	To prove \eqref{estx}, rewrite \eqref{fSDE0} in the integral form
	\begin{equation}\label{Exist1}
	y_{s,t} = \int_s^t f(y_u)du + \int_s^t g(y_u) dx_u.
	\end{equation}
	Together with \eqref{gcond} and \eqref{roughpvar}, we obtain
	\allowdisplaybreaks
	\begin{eqnarray*}
		&&\|y_{s,t}\| \\
		&\leq& \int_s^t\|f(y_u)\| du + \Big\|\int_s^t g(y_u) dx_u\Big\| \\
		&\leq& \int_s^t (C_f \|y_u\| + \|f(0)\|)du + C_g \ltn x \rtn_{\tp,[s,t]} + C_g^2 \ltn \X \rtn_{\tq,[s,t]^2} \\
		&&+ C_p \Big\{2C_g^2\ltn \X \rtn_{\tq,[s,t]^2}\ltn y \rtn_{\tp,[s,t]} \\
		&&+\ltn x \rtn_{\tp,[s,t]}\Big[C_g \ltn R^y \rtn_{\tq,[s,t]^2} + \frac{1}{2}C_g^2 \ltn x \rtn_{\tp,[s,t]} \ltn y \rtn_{\tp,[s,t]}\Big]  \Big\}	\\
		&\leq& \int_s^t (C_f \|y_u\| + \|f(0)\|)du + C_g \ltn x \rtn_{\tp,[s,t]} + C_g^2 \ltn \X \rtn_{\tq,[s,t]^2} \\
		&&+ C_p \Big\{\Big[2C_g^2\ltn \X \rtn_{\tq,[s,t]^2} + \frac{1}{2}C_g^2\ltn x \rtn_{\tp,[s,t]}^2\Big] \vee C_g \ltn x \rtn_{\tp,[s,t]} \Big\}\times\\
		&& \times 	 \Big( \ltn y \rtn_{\tp,[s,t]} + \ltn R^y \rtn_{\tq,[s,t]^2}\Big),
	\end{eqnarray*}
	which yields
	\begin{eqnarray*}
		&&\ltn y \rtn_{\tp,[s,t]} \\
		&\leq& \int_s^t (C_f \|y_u\| + \|f(0)\|)du + C_g \ltn x \rtn_{\tp,[s,t]} + C_g^2 \ltn \X \rtn_{\tq,[s,t]^2} \\
		&&+ C_p \Big\{\Big[2C_g^2\ltn \X \rtn_{\tq,[s,t]^2} + \frac{1}{2}C_g^2\ltn x \rtn_{\tp,[s,t]}^2\Big] \vee C_g \ltn x \rtn_{\tp,[s,t]} \Big\}\times \\
		&& \times  \Big( \ltn y \rtn_{\tp,[s,t]} + \ltn R^y \rtn_{\tq,[s,t]^2}\Big).
	\end{eqnarray*}
	By similar arguments, we can show that
	\begin{eqnarray*}
		&&\ltn R^y \rtn_{\tq,[s,t]^2} \\
		&\leq& \int_s^t (C_f \|y_u\| + \|f(0)\|)du + C_g^2 \ltn \X \rtn_{\tq,[s,t]^2} \\
		&&+ C_p \Big\{\Big[2C_g^2\ltn \X \rtn_{\tq,[s,t]^2} + \frac{1}{2}C_g^2\ltn x \rtn_{\tp,[s,t]}^2\Big] \vee C_g \ltn x \rtn_{\tp,[s,t]} \Big\}\times \\
		&& \times \Big( \ltn y \rtn_{\tp,[s,t]} + \ltn R^y \rtn_{\tq,[s,t]^2}\Big).
	\end{eqnarray*}
	Therefore by assigning $\ltn y,R \rtn_{\tp,[s,t]} =\ltn y \rtn_{\tp,[s,t]} + \ltn R^y \rtn_{\tq,[s,t]}$, we obtain
	\begin{eqnarray}\label{yRnorm}
	&&\ltn y,R^y \rtn_{\tp,[s,t]} \notag\\
	&\leq& 2\int_s^t (C_f \ltn y\rtn_{\tp,[s,u]} +C_f\|y_s\|+ \|f(0)\|)du \notag\\
	&& + C_g \ltn x \rtn_{\tp,[s,t]}  +2C_g^2 \ltn \X \rtn_{\tq,[s,t]^2} \notag\\
	&&+C_p \Big\{\Big[2C_g^2\ltn \X \rtn_{\tq,[s,t]^2} + \frac{1}{2}C_g^2\ltn x \rtn_{\tp,[s,t]}^2\Big] \vee C_g \ltn x \rtn_{\tp,[s,t]} \Big\} \times \notag\\
	&& \times \ltn y,R^y \rtn_{\tp,[s,t]}.
	\end{eqnarray}
	Observe that if $2C_p C_g \ltn \bx \rtn_{\tp,[s,t]} <1$ then
	\begin{eqnarray*}
	&&2C_p C_g \ltn \bx \rtn_{\tp,[s,t]} \\
	&>&  C_p \Big\{\Big[2C_g^2\ltn \X \rtn_{\tq,[s,t]^2} + \frac{1}{2}C_g^2\ltn x \rtn_{\tp,[s,t]}^2\Big] \vee C_g \ltn x \rtn_{\tp,[s,t]} \Big\}.
	\end{eqnarray*}
	This follows that
	\[
	\ltn y,R^y \rtn_{p{\rm -var},[s,t]} \leq \int_s^t 4 C_f \ltn y\rtn_{p{\rm -var},[s,u]} du + 4(\|f(0)\|+ C_f \|y_s\|)(t-s) +\frac{1}{C_p}
	\]
	whenever $2C_p C_g \ltn \bx \rtn_{\tp,[s,t]} \leq \frac{1}{2}$. Applying the continuous Gronwall lemma, we obtain
	\allowdisplaybreaks
	\begin{eqnarray}\label{yqvarest}
	\ltn y,R^y \rtn_{p{\rm -var},[s,t]} &\leq&  4(\|f(0)\|+ C_f \|y_s\|)(t-s) + \frac{1}{C_p} \notag\\
	&&+ \int_s^t 4C_f e^{4C_f (t-u)} \Big[4(\|f(0)\|+ C_f \|y_s\|)(u-s) + \frac{1}{C_p}\Big] du \notag\\
	&\leq& \Big(\frac{\|f(0)\|}{C_f}+ \frac{1}{C_p}+\|y_s\|\Big)e^{4C_f(t-s)} - \|y_s\|
	\end{eqnarray}
	whenever $4C_p C_g \ltn \bx \rtn_{\tp,[s,t]} \leq 1$. By constructing the greedy sequence of stopping times $\{\tau_i(\frac{1}{4C_pC_g},[a,b],p) \}$ and use similar estimates to the proof of \cite[Theorem 2.4]{duchong19}, we obtain \eqref{estx} and \eqref{estx2}.
	
\end{proof}

Theorem \ref{RDEnorm} enables us to prove the uniqueness of the solution of rough equation \eqref{fSDE0}.

\begin{theorem}[Uniqueness]\label{continuity}
	The solution $y_t(\bx,y_a)$ of \eqref{RDEgbounded} is uniformly continuous w.r.t. $y_a$. In particular, there exists a unique solution with respect to the initial condition $y_a$.
\end{theorem}
\begin{proof}
	The proof follows similar arguments as in Proposition \ref{prop1}. Indeed, let us consider two solutions $y_t(\bx,y_a)$ and $\ty_t(\bx,\ty_a)$ and their difference $z_t = \ty_t -y_t$, which satisfies the integral rough equation
	\[
	z_t = z_a +\int_a^t [f(\ty_s)-f(y_s)]ds+\int_a^t [g(\ty_s) - g(y_s)]dx_s.
	\]
	Then the estimate \eqref{roughest1} has the form
	\begin{eqnarray*}
		\|z_{s,t}\| &\leq& \Big\| \int_s^t [f(\ty_u)-f(y_u)]du\Big\| + \Big\| \int_s^t [g(\ty_u)-g(y_u)]dx_u\Big\| \notag\\
		&\leq& \int_s^t C_f \|z_u\|du +  C_g \|z_s\| \ltn x \rtn_{\tp,[s,t]} + 2C^2_g \|z_s\| \ltn \X \rtn_{\tq,[s,t]^2} \notag\\
		&& + C_p \Big\{\ltn x \rtn_{\tp,[s,t]} \ltn R^{g(\ty)-g(y)} \rtn_{\tq,[s,t]^2} 
		\\
		&&+ \ltn \X \rtn_{\tq,[s,t]^2} \ltn [g(\ty)-g(y)]^\prime\rtn_{\tp,[s,t]} \Big \},
	\end{eqnarray*}
	which, together with \eqref{roughest2}, \eqref{roughest3} and \eqref{estx2} yields
	\begin{eqnarray}\label{zRnorm}
	&&\ltn z,R^z \rtn_{\tp,[s,t]} \notag\\
	&\leq& 2\int_s^t C_f \|z_u\|du + 4C_p\Big\{C_g \ltn \bx \rtn_{\tp,[s,t]} \vee  C_g^2 \ltn \bx \rtn_{\tp,[s,t]}^2  \Big\} \times \notag \\
	&&\times  \Big(1 +\ltn \ty,R^{\ty} \rtn_{\tp,[s,t]} + \ltn y,R^y \rtn_{\tp,[s,t]} \Big) \Big(\|z_s\|+ \ltn z,R^z \rtn_{\tp,[s,t]} \Big)\notag\\
	&\leq& 2\int_s^t C_f \|z_u\|du + 4C_p\Big\{C_g \ltn \bx \rtn_{\tp,[s,t]} \vee  C_g^2 \ltn \bx \rtn_{\tp,[s,t]}^2  \Big\} \times \notag \\
	&&\times  \Big\{1 +2 \Big[\|y_a\| \vee \|\ty_a\|+ \Big(\frac{\|f(0)\|}{C_f} + \frac{1}{C_p}\Big) N_{[a,b]}(\bx)\Big]e^{4C_f(b-a) }N^{\frac{p-1}{p}}_{[a,b]}(\bx) \Big\}  \notag\\
	&& \times \Big(\|z_s\|+ \ltn z,R^z \rtn_{\tp,[s,t]} \Big).
	\end{eqnarray}	
	Provided that $\|y_a\|, \| \ty_a\| \leq r_0$, \eqref{zRnorm} has the form
	\begin{eqnarray*}
		\ltn z,R^z \rtn_{\tp,[s,t]} &\leq&  2\int_s^t C_f \Big(\|z_s\|+ \ltn z,R^z \rtn_{\tp,[s,u]} \Big)du \\
		&&+ 4C_p\Big\{C_g \ltn \bx \rtn_{\tp,[s,t]} \vee  C_g^2 \ltn \bx \rtn_{\tp,[s,t]}^2  \Big\}  \times\notag\\
		&& \times \Lambda(\bx,[a,b])\Big(\|z_s\|+ \ltn z,R^z \rtn_{\tp,[s,t]} \Big),
	\end{eqnarray*}
	where 
	\begin{eqnarray}
	&&\Lambda(\bx,[a,b]) = \Lambda(\bx,[a,b],r_0)\notag \\
	&=& 1 +2 \Big[r_0+ \Big(\frac{\|f(0)\|}{C_f} + \frac{1}{C_p}\Big) N_{[a,b]}(\bx)\Big]e^{4C_f(b-a) }N^{\frac{p-1}{p}}_{[a,b]}(\bx)>1.
	\end{eqnarray}
	Hence,
	\begin{equation}
	\|z_s\| + \ltn z,R^z \rtn_{\tp,[s,t]} \leq \|z_s\|+4\int_s^t C_f \Big(\|z_s\|+ \ltn z,R^z \rtn_{\tp,[s,u]} \Big)du 
	\end{equation}
	whenever $4C_p\Big\{C_g \ltn \bx \rtn_{\tp,[s,t]} \vee C_g^2 \ltn \bx \rtn_{\tp,[s,t]}^2  \Big\} \Lambda(\bx,[a,b]) \leq \frac{1}{2}$. Similar arguments to the proof of Theorem \ref{RDEnorm} then show that
	\begin{equation}\label{roughest4}
	\|z_a\|+\ltn z,R^z \rtn_{p{\rm -var},[a,b]} \leq \|z_a\|  e^{4C_f(b-a)+  \bar{N}_{[a,b]}(\bx)\log 2}\bar{N}_{[a,b]}(\bx)^{\frac{p-1}{p}}.
	\end{equation}
	where $\bar{N}_{[a,b]}(\bx)$ is the number of stopping times in the greedy sequence  
	\[
	\{\tau_i(\frac{1}{8C_p C_g \Lambda(\bx,[a,b])},[a,b],p) \},
	\]
	and could be estimated, using by \eqref{Nest}, as follows
	\begin{eqnarray*}
	\bar{N}_{[a,b]}(\bx) \leq \Big(1+ \Big[8C_p C_g \Lambda(\bx,[a,b])\Big]^p \ltn \bx \rtn_{p{\rm -var},[a,b]}^p \Big).
	\end{eqnarray*}
	This proves the continuity of the solution $y_t(\bx,y_a)$ on the initial condition $y_a$. In particular, when $\ty_a = y_a$, the right hand side of \eqref{roughest4} equals zero, which implies $\| \ty-y\|_{p{\rm -var},[a,b]} =0$, therefore the solution is unique given the initial condition $y_a$.
	
\end{proof}


\section*{Acknowledgments}
The author would like to thank the anonymous referee for her/his careful reading of the manuscript and insightful remarks which led to an improvement of the original submission. This work was supported by the Max Planck Institute for Mathematics in the Science (MIS-Leipzig). A part of this work was done while the author visited Vietnam Institute for Advanced Studies in Mathematics (VIASM) for a three month research stay in summer 2019.

\end{document}